\documentclass{amsart}

\usepackage[dvipsnames]{xcolor}
\usepackage{natbib}
\usepackage{amsmath}
\usepackage{amssymb}
\usepackage{amsthm}
\usepackage{bm}
\usepackage{mathtools}
\usepackage[colorlinks,pdfdisplaydoctitle,linkcolor=blue,citecolor=red]{hyperref}
\usepackage{url,subcaption}
\usepackage[inline]{enumitem}

\def\uvs{\bm{s}}
\def\uvt{\bm{t}}
\def\uvx{\bm{x}}
\def\uvy{\bm{y}}
\def\uvz{\bm{z}}
\def\uvtheta{\bm{\theta}}
\def\uvphi{\bm{\phi}}
\def\uvpsi{\bm{\psi}}

\DeclareMathOperator*{\argmax}{arg\,max}
\DeclareMathOperator*{\argmin}{arg\,min}
\DeclareMathOperator{\Conv}{Conv}
\newcommand{\norm}[1]{\left\lVert#1\right\rVert}
\def\floor#1{\left\lfloor{#1}\right\rfloor}
\def\GGiven{\,\bigg|\,}

\def\given{\,|\,}

\numberwithin{equation}{section}

\theoremstyle{plain}
\newtheorem{thm}{Theorem}[section]

\newtheorem{prop}[thm]{Proposition}

\theoremstyle{definition}
\newtheorem{defn}[thm]{Definition}
\newtheorem{example}[thm]{Example}

\theoremstyle{remark}
\newtheorem{remark}[thm]{Remark}

\title{Elements of asymptotic theory with outer probability measures}

\author[J. Houssineau]{Jeremie Houssineau}
\address{Department of Statistics, University of Warwick, Coventry, CV4 7AL, UK}
\email{jeremie.houssineau@warwick.ac.uk}

\author[N. K. Chada]{Neil K. Chada}
\address{Department of Statistics and Applied Probability, National University of Singapore, 119077, Singapore}
\email{neil.chada@nus.edu.sg}

\author[E. Delande]{Emmanuel Delande}
\address{Cockrell School of Engineering, University of Texas at Austin, Austin, 78712, USA}
\email{manu-delande@gmail.com}

\begin{document}

\maketitle

\begin{abstract}
Outer measures can be used for statistical inference in place of probability measures to bring flexibility in terms of model specification. The corresponding statistical procedures such as Bayesian inference, estimators or hypothesis testing need to be analysed in order to understand their behaviour, and motivate their use. In this article, we consider a class of outer measures based on the supremum of particular functions that we refer to as possibility functions. We then characterise the asymptotic behaviour of the corresponding Bayesian posterior uncertainties, from which the properties of the corresponding maximum a posteriori estimators can be deduced. These results are largely based on versions of both the law of large numbers and the central limit theorem that are adapted to possibility functions. Our motivation with outer measures is through the notion of uncertainty quantification, where verification of  these procedures is of crucial importance. These introduced concepts shed a new light on some standard concepts such as the Fisher information and sufficient statistics and naturally strengthen the link between the frequentist and Bayesian approaches.
\end{abstract}

\section{Formulation}
\label{sec:form}

The general objective of statistical inference is to find the true value of a parameter of interest given some observed data. The set of all possible parameters values is denoted $\Theta$. It is assumed that the observed data, denoted $y$, is the realisation of a random variable $Y$ on some observation space $\mathcal{Y}$ which is related to the parameter $\theta \in \Theta$ via a parametric family of probability distributions $\{p_{Y}(\cdot\,; \theta)\}_{\theta \in \Theta}$, often referred to as the likelihood. Different estimators for the true value of the parameter $\theta_0$ can then be considered, but the most common is the maximum likelihood estimator (MLE) \citep{AH99,JA97}, defined as
\begin{equation}
\label{eq:generalMle}
\hat{\theta} = \argmax_{\theta \in \Theta} p_Y(y; \theta).
\end{equation}
In the frequentist approach \citep{CB01,RAF22}, the uncertainty about the value of $\theta_0$ can be quantified via a confidence interval which is computed for a given confidence level. It is however important to note that it is the interval itself that is random (as a function of the observation $Y$) so that it is the interval that has a coverage probability, i.e.\ a probability of containing the true parameter, instead of the parameter having some probability to be contained in the interval. From a computational viewpoint, the MLE is often relatively easy to compute; however, confidence intervals are usually more difficult to deal with and offer a limited understanding of the uncertainty about the parameter. 

The principle of the Bayesian approach \citep{RC13} is to infer the posterior probability distribution of a random variable of interest given the observed data. Even if the quantity of interest is fixed, the uncertainty about this quantity is modelled as a random variable $X$ on some space $\mathcal{X}$ related to $\Theta$, e.g.\ $\mathcal{X} \supseteq \Theta$. One can then represent prior knowledge about $X$ as a probability distribution $p_X$ on $\mathcal{X}$ and then, through Bayes' formula, characterise the posterior distribution $p_{X|Y}(\cdot \given y)$ of $X$ as
\begin{equation}
\label{eq:Bayes}
p_{X|Y}(x \given y) = \frac{p_{Y|X}(y\given x) p_X(x)}{\int_{\mathcal{X}} p_{Y|X}(y\given z) p_X(z) \mathrm{d} z},
\end{equation}
for any $x \in \mathcal{X}$, where the likelihood is now a conditional probability distribution $p_{Y|X}(\cdot \given x)$ defined for any $x \in \mathcal{X}$. The integral in the denominator of \eqref{eq:Bayes} is called the marginal likelihood. In this context, the analogue of the MLE \eqref{eq:generalMle} is the maximum a posterior (MAP) estimate $\hat{x} = \argmax_{x \in \mathcal{X}} p_{X|Y}(x \given y)$. Bayesian inference is generally more computationally demanding than frequentist inference but offers a full characterisation of the posterior uncertainty. For a given confidence level, a credible interval can be calculated. In this interpretation, it is the random variable $X$ that is contained in the fixed credible interval with a certain probability. The marginal likelihood is also a useful quantity as it reflects the coherence between the observation and the considered modelling. It appears naturally in hierarchical models where it can be interpreted as a likelihood for the higher-order parameters. However, it is often difficult to compute and most of the techniques in computational statistics avoid the explicit calculation of this term.

In spite of the ever increasing available computational resources, the determination of the posterior probability distribution can be difficult to achieve due to the complexity of the model, to the amount of data or to physical constraints such as in real-time applications. Furthermore, discrepancies between the model and the actual mechanisms underlying real data can be at the origin of a range of issues, from the unreliability of the computed posterior distributions to the divergence of the considered algorithm.

\subsection{Proposed approach}
\label{sec:approach}

The purpose of this article is to propose an alternative representation of uncertainty that will lead to
\begin{enumerate*}[label=\arabic*)]
\item the strengthening of the connection between the frequentist and Bayesian approaches, with the objective of providing a more pragmatic computational framework, and to
\item the possibility of a less conservative modelling, with the objective of improving the overall robustness and decreasing the sensitivity to misspecification.
\end{enumerate*}

We follow the same premise as possibility theory \citep{DP15}, which can be motivated as follows. As in the frequentist approach, we consider that the unknown true parameter $\theta_0 \in \Theta$ of a statistical model is not random and, therefore, the uncertainty about this parameter might not be best modelled by a probability distribution on the set $\Theta$. When representing deterministic uncertainty with a probability measure $p$, the most stringent assumption is additivity. Indeed, if there is no available information regarding $\theta_0$ as an element of $\Theta$, one might not be able to define a meaningful credibility for the event $\theta_0 \in A$ for some $A \subseteq \Theta$ in a way that verifies $p(A) + p(A^{\mathrm{c}}) = 1$ with $A^{\mathrm{c}}$ the complement of $A$ in $\Theta$. Instead, it would be convenient to have a measure of credibility for which both $A$ and $A^{\mathrm{c}}$ could have a credibility equal to $1$ if there is no objection against these events. When removing the assumption of additivity from the concept of probability measure, one obtains the concept of outer measure. Since we still want the corresponding set function, say $\bar{P}$, to give value $1$ to the entire parameter set $\Theta$, we call $\bar{P}$ an \emph{outer probability measure} (o.p.m.). The corresponding type of uncertainty, which relates to non-random phenomena, is often referred to as \emph{epistemic} uncertainty \citep{GS61,PW91}. The two main differences between o.p.m.s and probability measures are that
\begin{enumerate*}[label=\arabic*)]
\item they are sub-additive, i.e.\ $\bar{P}(A \cup B) \leq \bar{P}(A) + \bar{P}(B)$ for any $A,B \subseteq \Theta$ and the equality does not have to hold even if $A$ and $B$ are disjoint, and 
\item they can be evaluated on all subsets, hence avoiding some measure-theoretic technicalities even though they have been mostly used as a measure-theoretic tool, e.g.\ for constructing the Lebesgue measure.
\end{enumerate*}
In particular, up to measurability conditions, a probability measure can be seen as a special case of an o.p.m.

In the case where there is no information about whether $\theta_0$ is in $A$ or not, one would like to set $\bar{P}(A) = \bar{P}(A^{\mathrm{c}}) = 1$. Seeing the value $p(A)$ that a probability measure $p$ gives to a subset $A$ as an integral of the corresponding p.d.f., also denoted by $p$, over the set $A$, i.e.\ $p(A) = \int_A p(\theta) \mathrm{d} \theta$, it is natural to seek an analogous expression for an o.p.m.\ $\bar{P}$ on $\Theta$. An operator allowing for $\bar{P}(A \cup A^{\mathrm{c}}) = \bar{P}(A) = \bar{P}(A^{\mathrm{c}}) = 1$ is the maximum, or, more generally the supremum. We therefore consider $\bar{P}(A) = \sup_{\theta \in A} f(\theta)$, with $f$ a suitable function. The assumptions on $\bar{P}$ yield the constraints $f \geq 0$ and $\sup_{\theta \in \Theta} f(\theta) = 1$. Although these functions are called possibility distributions in possibility theory, we refer to them as \emph{possibility functions} so as to better distinguish them from probability distributions.

In order to characterise the relations between different unknown quantities of interest, it is useful to introduce an analogue of the concept of random variable as follows: let $\Omega_{\mathrm{u}}$ be the sample space for deterministic but uncertain phenomena, then an \emph{uncertain variable}\footnote{The concept of uncertain variable as introduced by \citet{JH18} includes both random and deterministic forms of uncertainty; in that context, uncertain variables as introduced here would be called \emph{deterministic uncertain variables}} on $\Theta$ is a surjective mapping\footnote{Assuming that the mapping is surjective is not a limitation since $\Omega_{\mathrm{u}}$ can always be made large enough to satisfy it.} $\uvtheta$ from $\Omega_{\mathrm{u}}$ to $\Theta$. The main differences with a random variable are that the sample space is not equipped with a $\sigma$-algebra and a probability measure, and there is no measurability condition on the mapping; instead, there is a true ``state of nature'' $\omega_{\mathrm{u}}^* \in \Omega_{\mathrm{u}}$ which is such that $\uvtheta(\omega_{\mathrm{u}}^*)$ is the true value of the parameter. Yet, the concepts of realisation and event are meaningful for uncertain variables. Since possibility functions only model information, it follows that an uncertain variable does not induce a unique possibility function. Instead, different possibility functions represent different levels of knowledge about an uncertain variable. For this reason, we say that a possibility function \emph{describes} an uncertain variable. Analogous constructions have been considered in \citep{DMD99,DMD00}, where a connection between probability theory and control theory has been made in the context of $(\max,+)$ algebras \citep{PB10,VM92}, as well as in \citep{PT14}, where a law of large numbers is derived for a fuzzy-set-related notion of mean value \citep{CF01}. \citet{CD17} also considered extending statistical concepts to non-random variables in the context of causal inference.

The ultimate goal of the proposed approach is to allow for inference to be performed in the presence of both deterministic and random sources of uncertainty so that the different components of complex statistical models can be modelled as faithfully as possible. One way of achieving this goal is to consider more general o.p.m.s of the form
$$
\bar{P}_{\uvtheta,Y}(A \times B) = \sup_{\theta \in A} f_{\uvtheta}(\theta) \int_B p_{Y|\uvtheta}(y \given \theta) \mathrm{d} y,
$$
with $A$ a subset of some given set $\Theta$ and $B$ a measurable subset of $\mathcal{Y}$, where $\uvtheta$ represents the uncertainty about the true parameter $\theta_0$ in $\Theta$ and is described by a possibility function $f_{\uvtheta}$ on $\Theta$ and where $Y$ is a random variable on $\mathcal{Y}$ characterising the random components of the model with law $p_{Y|\uvtheta}(\cdot \given \uvtheta)$. These o.p.m.s are studied more formally in \citet{JH18} and their application in the context of Bayesian inference for some complex systems is considered in \citet{JH18b}. Applying Bayes' theorem to the o.p.m.\ $\bar{P}_{\uvtheta,Y}$ yields
\begin{equation}
\label{eq:likelihoodAsProbability}
f_{\uvtheta|Y}(\theta \given y) = \dfrac{p_{Y|\uvtheta}(y \given \theta) f_{\uvtheta}(\theta)}{\sup_{\psi \in \Theta} p_{Y|\uvtheta}(y \given \psi)f_{\uvtheta}(\psi)}.
\end{equation}
When the prior possibility function $f_{\uvtheta}$ is equal to the indicator function of $\Theta$, denoted $\bm{1}_{\Theta}$, which corresponds to the case where there is no information about $\uvtheta$, this posterior possibility function can be written as $f_{\uvtheta|Y}(\theta \given y) = p_{Y|\uvtheta}(y \given \theta)/p_{Y|\uvtheta}(y \given \hat{\theta})$ with $\hat{\theta}$ the MLE. The idea of considering quantities of the form \eqref{eq:likelihoodAsProbability} has been discussed in the literature, see for instance \citep{YYC95,WM99}; however, some of the theoretical foundations backing this approach are lacking.

\begin{remark}
The quantification of uncertainty in frequentist inference is directly related to the true value of the parameter but can only be interpreted across different realisations of the observation. Bayesian posterior probability distributions focus on one realisation of the observation but assume the parameter to be random. The posterior quantification of uncertainty provided by $f_{\uvtheta|Y}(\cdot \given y)$ focuses on one realisation of the observation \emph{and} on the true value of the parameter. An important consequence is that greater prior uncertainty would typically reduce the posterior probability of a given event in the standard Bayesian framework whereas the posterior credibility of the same event will tend to increase in the same situation when using possibility functions. This aspect is particularly important in decision theory \citep{GD04,JQS10}.
\end{remark}

In general, we might want to include prior information and consider an informative prior possibility function $f_{\uvtheta}$, that is $f_{\uvtheta}(\theta) < 1$ for some $\theta \in \Theta$. One way to obtain such possibility functions is to simply renormalise a bounded probability distribution, e.g.\ the normal possibility function with parameters $\mu \in \Theta \subseteq \mathbb{R}$ and $\sigma^2 > 0$ is defined as
\begin{equation}
\label{eq:normalPossibility}
\overline{\mathrm{N}}(\theta; \mu, \sigma^2) = \exp\Big( -\dfrac{1}{2\sigma^2} (\theta - \mu)^2 \Big).
\end{equation}
We will show in Section~\ref{sec:CLT} that $\mu$ and $\sigma^2$ can be rightfully referred to as the expected value and the variance respectively. For the sake of simplicity, we will write $\overline{\mathrm{N}}(\mu, \sigma^2)$ when referring to the function $\overline{\mathrm{N}}(\cdot\,; \mu, \sigma^2)$. It would be more natural to parametrise the normal possibility function by the precision $\tau = 1/\sigma^2$ since the case $\tau = 0$ is well-defined and corresponds to the uninformative possibility function equal to $1$ everywhere; yet, the parametrisation by the variance is considered for the sake of consistency with the probabilistic case.

It is important to extend widely-used probabilistic concepts to uncertain variables: consider two uncertain variables $\uvpsi$ and $\uvphi$ on the respective spaces $\Psi$ and $\Phi$ and assume that these uncertain variables are jointly described by the possibility function $f_{\uvpsi,\uvphi}$ on $\Psi \times \Phi$, then $\uvpsi$ and $\uvphi$ are said to be \emph{independently described} if there exists possibility functions $f_{\uvpsi}$ and $f_{\uvphi}$ such that $f_{\uvpsi,\uvphi}(\psi,\phi) = f_{\uvpsi}(\psi) f_{\uvphi}(\phi)$ for any $(\psi,\phi) \in \Psi \times \Phi$. This notion of independence models that the information we hold about $\uvpsi$ is not related to $\uvphi$ and conversely. Another connection with standard statistical techniques can be made via the \emph{profile likelihood} \citep{MV00} using the change of variable formula for possibility functions. Indeed, if $\uvtheta$ is an uncertain variable on $\Theta$ described by $f_{\uvtheta}$ then, for any mapping $\zeta : \Theta \to \Psi$, the uncertain variable $\uvpsi = \zeta(\uvtheta)$ can be described by
\begin{equation}
\label{eq:changeOfVariable}
f_{\uvpsi}(\psi) \doteq \sup \{ f_{\uvtheta}(\theta) : \theta \in \zeta^{-1}[\psi] \},
\end{equation}
for any $\psi \in \Psi$, where we can ensure that the inverse image $\zeta^{-1}[\cdot]$ is non-empty by assuming that $\zeta$ is surjective, otherwise the appropriate convention is $\sup \emptyset = 0$. Unlike the change of variable formula for p.d.f.s, \eqref{eq:changeOfVariable} does not contain a Jacobian term since $f_{\uvtheta}$ is not a density even when, say, $\Theta = \mathbb{R}$. In the case where $\Theta = \Psi \times \Phi$ and $\zeta(\psi, \phi) = \psi$, we find that $f_{\uvpsi}(\psi) \doteq \sup_{\phi \in \Phi} f_{\uvtheta}(\psi,\phi)$, which is the analogue of marginalisation \citep{LZ78}. This operation is often used when the number of parameters is too high and one wants to remove \emph{nuisance} parameters. The motivation behind profile likelihood is shown here to be consistent with the general treatment of possibility functions, as detailed below in Example~\ref{ex:profileLikelihood}. There are existing results on the identifiability analysis \citep{RKM09} and uncertainty analysis \citep{VTHR12} associated with profile likelihoods. If $\uvtheta$ and $\uvpsi$ are two uncertain variables on $\mathbb{R}^d$, jointly described by the possibility function $f_{\uvtheta,\uvpsi}$ then, using the change of variable formula \eqref{eq:changeOfVariable}, one can show that the uncertain variable $\alpha \uvtheta + \uvpsi$, for any scalar $\alpha \neq 0$, is described by
\begin{equation}
\label{eq:multScalar}
f_{\alpha \uvtheta + \uvpsi}(\phi) = \sup \big\{ f_{\uvtheta,\uvpsi}(\theta,\psi) : \theta, \psi \in \mathbb{R}^d,\, \alpha \theta + \psi = \phi \big\}.
\end{equation}
This formula is well known in possibility theory, see e.g.\ \citep{DP81}. From a computational viewpoint, the expression of the possibility function $f_{\alpha\uvtheta+\uvpsi}$ describing $\alpha\uvtheta+\uvpsi$ is simpler than in the probabilistic case. We can easily show via \eqref{eq:multScalar} that the normal possibility function shares some of the properties of its probabilistic analogue. Indeed, if $\uvtheta$ is an uncertain variable described by the possibility function $\overline{\mathrm{N}}(\mu, \sigma^2)$ then for any scalar $\alpha > 0$, the uncertain variable $\alpha \uvtheta$ is described by $\overline{\mathrm{N}}(\alpha \mu, (\alpha \sigma)^2)$.

The concept of uncertain variable is also useful in the context of hypothesis testing because the different hypotheses, say, $\mathrm{H}_0 : \uvtheta = \theta_0$ versus $\mathrm{H}_1 : \uvtheta \neq \theta_0$, corresponds to events which have non-zero credibility in general, e.g.\ the credibility of the event $\uvtheta = \theta_0$ when $\uvtheta$ is described by the possibility function $f_{\uvtheta}$ is $f_{\uvtheta}(\theta_0)$. Discussions about hypothesis testing can be found in Section~\ref{sec:LRT_unknown}.

This article aims to exploit further the principles introduced so far to define a likelihood for possibility functions, derive its usual asymptotic properties, and highlight the consequences of such modelling for estimation purposes. Therefore, we consider a conditional possibility function $f_{\uvy|\uvtheta}(\cdot\given \theta)$, $\theta \in \Theta$, which describes the observation process, and we model the uncertainty about the true value of the parameter by an uncertain variable $\uvtheta$ on $\Theta$. For a given possibility function $f_{\uvtheta}$ modelling the prior knowledge about $\uvtheta$, the associated posterior possibility function is
\begin{equation}
\label{eq:likelihoodAsPossibility}
f_{\uvtheta|\uvy}(\theta \given y) = \dfrac{f_{\uvy|\uvtheta}(y \given \theta) f_{\uvtheta}(\theta)}{\sup_{\psi \in \Theta} f_{\uvy|\uvtheta}(y\given \psi)f_{\uvtheta}(\psi)}.
\end{equation}
This form of conditioning follows from applying Bayes' rule in the context of possibility theory \citep{dBTM99}. The marginal likelihood
$$
f_{\uvy}(y) = \sup_{\psi \in \Theta} f_{\uvy}(y \given \psi)f_{\uvtheta}(\psi)
$$
is always a dimensionless scalar in the interval $[0,1]$ which can be easily interpreted as the degree of coherence between the model and the data. This is not the case when using the likelihood $p_{Y|\uvtheta}(\cdot \given \theta)$ in general. This advantage, however, does not come for free since a value of $f_{\uvy}(y)$ that is close to one does not imply that the considered model is a good model, it only implies that the observation $y$ is compatible with the model. For instance, if there is no information about the observation process, then one can set $f_{\uvy|\uvtheta}(\cdot \given \theta) = \bm{1}_{\mathcal{Y}}$, in which case any observation will receive the maximal marginal likelihood, i.e.\ $f_{\uvy}(y) = 1$ for any $y \in \mathcal{Y}$. Although this might make the implementation of some tasks like model selection more challenging, it could also make other tasks such as checking for prior-data conflicts \citep{EM06} more straightforward. With these additional notations, we can now come back to the topic of profile likelihood in the following example.

\begin{example}
\label{ex:profileLikelihood}
If there is no prior knowledge about $\Theta$, then one can set $f_{\uvtheta} = \bm{1}_{\Theta}$ so that the posterior possibility function simplifies to
$$
f_{\uvtheta|\uvy}( \theta \given y) = \dfrac{f_{\uvy|\uvtheta}(y \given \theta)}{\sup_{\theta' \in \Theta} f_{\uvy|\uvtheta}(y\given \theta')} \propto f_{\uvy|\uvtheta}(y \given \theta),
$$
which can be seen as an inversion of the conditioning between $\uvtheta$ and $\uvy$. If the uncertain variable $\uvtheta$ is of the form $(\uvpsi,\uvphi)$, with $\uvpsi$ and $\uvphi$ some uncertain variables on $\Psi$ and $\Phi$ respectively, and if we marginalise the variable $\uvphi$, then we find that the marginal posterior possibility function describing $\uvpsi$ is
$$
f_{\uvpsi|\uvy}(\psi \given y) = \sup_{\phi \in \Phi} \dfrac{f_{\uvy|\uvtheta}(y \given \psi, \phi)}{\sup_{\theta \in \Theta} f_{\uvy|\uvtheta}(y\given \theta)}.
$$
Inverting the conditioning once again, we obtain the likelihood for $\uvpsi$ only as
\begin{equation}
\label{eq:ex:profileLikelihood}
f_{\uvy|\uvpsi}(y \given \psi) \propto \sup_{\phi \in \Phi} f_{\uvy|\uvtheta}(y \given \psi, \phi),
\end{equation}
which justifies the maximisation over nuisance parameters in the profile likelihood. This also shows that care must be taken when considering the marginal likelihood 
since \eqref{eq:ex:profileLikelihood} contains a normalising constant.
\end{example}

In terms of interpretation, we do not view the possibility function $f_{\uvy|\uvtheta}(\cdot \given \theta)$ as an upper bound for the true likelihood as suggested in \citep{DMP95}; instead, $f_{\uvy|\uvtheta}(\cdot \given \theta)$ simply relates some characteristics of the data generating process of interest, expressed via the parameter $\theta$, to the observation $y$. This aspect will be further developed in the following sections. In the situation where there are several data points $y_1,\dots,y_n$, we proceed consistently with the usual treatment and assume that these data points are realisations of uncertain variables $\uvy_1,\dots,\uvy_n$ that are conditionally-independently described by $f_{\uvy|\uvtheta}(\cdot \given \theta)$ given $\uvtheta = \theta$. It follows that the associated posterior possibility function takes the form
$$
f_{\uvtheta|\uvy_{1:n}}( \theta \given y_1,\dots,y_n) = \dfrac{\prod_{i=1}^n f_{\uvy|\uvtheta}(y_i \given \theta) f_{\uvtheta}(\theta)}{\sup_{\psi \in \Theta} \prod_{i=1}^n f_{\uvy|\uvtheta}(y_i\given \psi)f_{\uvtheta}(\psi)},
$$
where $\uvy_{1:n}$ stands for the sequence $(\uvy_1,\dots,\uvy_n)$. The practical motivation for considering a conditional possibility function as a likelihood comes from the need to derive equivalent tools for understanding the asymptotic behaviour of estimators when there is little knowledge about the true form of the distribution of the observations, as is often the case with real data. The analysis of statistical techniques expressed in this formalism will be conducted by simply relying on $f_{\uvy|\uvtheta}(\cdot \given \theta)$. 

We have considered two possible models for the likelihood: as a probability distribution \eqref{eq:likelihoodAsProbability} or as a possibility function \eqref{eq:likelihoodAsPossibility}. In fact any conditional o.p.m.\ could be used as a likelihood. Crucially, the nature of the likelihood has no bearing on the nature of the posterior uncertainty. Indeed, if the prior is a possibility function then so will be the posterior. In particular, this highlights that although the information contained in the prior might be forgotten as the number of observations increases, the nature of the prior as a probability distribution or as a possibility function is never forgotten. The nature of the prior should therefore be chosen at least as carefully as its shape.

The idea of replacing some or all of the probabilistic ingredients in statistical inference is not new. Fiducial inference as introduced by \citet{RAF35} is probably one of the first examples of such an approach. M-estimators \citep{VPG91} and quasi-likelihoods \citep{RWMW74} are also attempts at replacing the standard likelihood model. Alternatives to the standard approach can also be introduced using belief functions \citep{APD68,GS76,PS94}. Recently, \citet{BHW16} proposed the use of exponentiated loss functions as likelihoods in a Bayesian inference framework. The idea of using upper bounds in order to bring flexibility in the Bayesian approach is also common, for instance, the so-called provably approximately correct (PAC) Bayes method \citep{JL05} aims to minimise the upper bound for a given loss function. In spite of these connections, the proposed approach differs in many aspects from the existing literature and, to the best of the authors' knowledge, the asymptotic properties that are derived in this article are novel. We review some of the features of the generalised Bayesian inference framework proposed by \citet{BHW16} in the next section.

\subsection{Generalised Bayesian inference}

The main contribution of \citet{BHW16} is to motivate the use of a loss function $\mathcal{L}(\theta,y)$ to describe the relation between the data $Y$ sampled from a true distribution $p_0$ on a set $\mathcal{Y}$ and a parameter of interest $\theta$ in a set $\Theta$. It is then demonstrated by \citet{BHW16} from two different approaches that
$$
p(\theta \given y) = \dfrac{\exp(-\mathcal{L}(\theta,y))p(\theta)}{\int \exp(-\mathcal{L}(\theta',y))p(\theta') \mathrm{d} \theta'},
$$
is a consistent/optimal update rule for the subjective prior probability $p$ given a realisation $y$ of $Y$. In this, situation, the ``true'' parameter is argued to be the minimiser of the expected loss $\mathbb{E}(\mathcal{L}(\theta, Y))$. In order to calibrate the losses with respect to the prior distributions and to the data, it is suggested by \citet{BHW16} that an additional term $\log(p(\hat{\theta}))$ with $\hat{\theta} = \max_{\theta \in \Theta} p(\theta)$ can be used to ensure non-negativity and $\mathcal{L}$ can be standardised by assuming that $\min_{\theta \in \Theta} \mathcal{L}(\theta,y) = 0$ for any $y \in \mathcal{Y}$, leading to an overall loss of the form
$$
\bar{\mathcal{L}}(\theta; y, p) = \beta \mathcal{L}(\theta, y) + \log(p(\hat{\theta})/p(\theta)),
$$
with $\beta > 0$ an annealing-type inverse temperature controlling the influence of the data point $y$ on the posterior probability distribution.

The motivation behind this generalised Bayesian inference framework is very similar to the one underlying the introduction of possibility functions. In fact, the information provided by a loss function as well as the subjective beliefs held a priori are often of a non-stochastic nature. Therefore, assuming that possibility functions, and o.p.m.s in general, provide a suitable framework for manipulating information, it is sensible to use this framework together with the exponentiated loss functions studied by \citet{BHW16}. Moreover, exponentiating the reward $-\bar{\mathcal{L}}(\cdot\,; y, p)$ yields
$$
\exp\big(-\bar{\mathcal{L}}(\theta; y, p)\big) = \exp\big(-\mathcal{L}(\theta,y)\big)^{\beta} \dfrac{p(\theta)}{p(\hat{\theta})},
$$
where $\exp(-\mathcal{L}(\cdot,y))$ and $p(\cdot)/p(\hat{\theta})$ can be identified with the likelihood $f_{\uvy|\uvtheta}(y \given \cdot)$ and with the prior $f_{\uvtheta}$ respectively. It holds that $f_{\uvy|\uvtheta}(\cdot \given \theta)$ to the power $\beta$ remains a possibility function for any $\theta \in \Theta$ and it will appear in further sections that this operation preserves the expected value of the underlying uncertain variable while dividing its variance by $\beta$. The case of non-stochastic data is also mentioned by \citet{BHW16}, which furthers the connection with the proposed approach. Seeing the likelihood $f_{\uvy|\uvtheta}(y \given \cdot)$ as being related to a loss function will provide insights in the theory developed in further sections and will facilitate the interpretation of some of the results.

\subsection{Structure and contributions}

There are a number of original contributions in this article and we highlight the most important ones here:
\begin{enumerate}[label=\arabic*), wide, nosep]
\item In the continuation of Section~\ref{sec:approach}, we introduce the analogues of a number of statistical concepts in Section~\ref{sec:notions} and highlight how some of these notions bring additional insight in modelling and inference.
\item A law of large numbers (LLN) and a central limit theorem (CLT) are derived for uncertain variables in Section~\ref{sec:CLT}. These results lead to the introduction of meaningful definitions of expected value and variance for uncertain variables and confirm the fundamental role of the normal possibility function in the proposed approach. In particular, the considered notion of expected value is related to the mode and, therefore, to the MAP in a Bayesian setting. The variance is also shown to be strongly related to the notion of Fisher information. These results allow for defining notions of identifiability and consistency.
\item Asymptotic properties of the posterior possibility function are studied Section~\ref{sec:MLE_un} using the available information only. This is crucial in applications since the true sampling distribution is unavailable. The connection between Bayesian inference and frequentists techniques such as likelihood ratio tests (LRTs) is then strengthened by showing that their asymptotic behaviour can be deduced from the Bernstein-von Mises (BvM) theorem.
\end{enumerate}

\begin{table}
\caption{Recurring notations}
\label{tbl:notations}
\renewcommand\arraystretch{1.25}
\begin{tabular}{r | l | c}
Notation & Description & Equation \\
\hline
$\Omega_{\mathrm{u}}$, $\bar{\mathbb{P}}$ & Sample space for uncertain variables and o.p.m. &  \\
$\uvx$, $\uvy$, $\uvtheta$ & Uncertain variables & \\
$\mathrm{N}(\mu, \sigma^2)$ & Normal distribution & \\
$\overline{\mathrm{N}}(\mu, \sigma^2)$ & Normal possibility function & \eqref{eq:normalPossibility} \\
$\mathbb{E}^*(\cdot)$, $\mathbb{V}^*(\cdot)$ & Expectation and variance w.r.t.\ $\bar{\mathbb{P}}$ & \eqref{eq:expectation}, \eqref{eq:variance}  \\
$\xrightarrow{o.p.m.}$ & Convergence in o.p.m. & \eqref{eq:convOpm} \\
$\mathcal{J}^*_n$ & Observed information based on $n$ observations & \eqref{eq:obsInfo} \\
$\mathcal{I}^*(\theta)$ & Fisher information based on $\mathbb{E}^*(\cdot \given \theta)$ & \eqref{eq:Fisher}
\end{tabular}
\end{table}

\subsection{Notations}

Deriving elements of asymptotic theory in a new context requires the introduction of a number of notations. These new notations are described in Table~\ref{tbl:notations} together with their standard counterparts.
Unless stated otherwise, the considered sets $\mathcal{X}$, $\mathcal{Y}$, $\mathcal{Z}$ and $\Theta$ are all assumed to subsets of the real line. Partial derivatives w.r.t.\ $x$, $y$, $z$ and $\theta$ will be denoted by $\partial_x$, $\partial_y$, $\partial_z$ and $\partial_{\theta}$ respectively.

\section{Statistical modelling with o.p.m.s}
\label{sec:notions}

Standard statistical inference uses a range of notions and techniques for the purpose of modelling, inference and analysis which are mostly based on probabilistic concepts. Introducing an alternative representation of uncertainty means that most, if not all, of these fundamental notions and techniques must be redefined. We show in this section that this can however be done easily without loosing the associated interpretations.

\paragraph{Location and scale parameters.}

We consider a likelihood function of the form $f_{\uvy|\uvtheta}(\cdot \given \theta)$. The parameter $\theta$ is said to be a \emph{location parameter} if there exists a possibility function $f_{\mathrm{loc}}$ on $\mathcal{Y}$ such that $f_{\uvy|\uvtheta}(y \given \theta) = f_{\mathrm{loc}}(y - \theta)$. Similarly, the parameter $\theta$ is said to be a \emph{scale} parameter if there exists a possibility function $f_{\mathrm{sca}}$ on $\mathcal{Y}$ such that $f_{\uvy|\uvtheta}(y \given \theta) = f_{\mathrm{sca}}(y/\theta)$. Rate and shape parameters can then be defined as usual. The possibility function $f_{\uvy|\uvtheta}(\cdot \given \theta)$ can be seen as providing an upper bound for subjective probability distributions related to $\uvy$ as follows: the probability measure $p( \cdot \given \theta)$ is upper-bounded set-wise by $f_{\uvy|\uvtheta}(\cdot \given \theta)$ if
$$
p(B \given \theta) \leq \sup_{y \in B} f_{\uvy|\uvtheta}(y \given \theta), \qquad B \subseteq \mathcal{Y}, \theta \in \Theta.
$$
In particular, if $f_{\uvy|\uvtheta}(\cdot \given \theta)$ is monotone increasing then this set-wise upper bound can be rephrased as $F(y \given \theta) \leq f_{\uvy|\uvtheta}(y \given \theta)$ for any $y \in \mathcal{Y}$ and any $\theta \in \Theta$ with $F(\cdot \given \theta)$ the cumulative distribution function (c.d.f.) induced by $p(\cdot \given \theta)$. This shows that possibility functions behave more closely to c.d.f.s than to p.d.f.s, as already hinted at by the definition of a scale parameter for possibility functions. 

\paragraph{Conjugate prior.}

The concept of conjugate prior can be extended straightforwardly to possibility functions whether the likelihood is a p.d.f.\ or a possibility function. In the former case, the renormalised version of any conjugate probability distribution is also a conjugate prior as a possibility function with the same parameters. For instance, the normal possibility function is conjugate for the expected value of a normal likelihood.\footnote{in fact, a Kalman filter can be derived as demonstrated in \citet{HB18}} Another example is the gamma possibility function on $[0,\infty)$ defined as
\begin{equation*}
\overline{\mathrm{G}}(\theta; \alpha, \beta) = \dfrac{(\beta\theta)^{\alpha}}{\alpha^{\alpha}} \exp\big( \alpha - \beta \theta \big),
\end{equation*}
with shape parameter $\alpha \geq 0$ and rate parameter $\beta \geq 0$, where $\alpha^\alpha$ is assumed to be equal to $1$ if $\alpha = 0$. Note that this is not simply the renormalized version of the gamma probability distribution: the interval of definition of the shape parameter has been shifted by $-1$ and the value $\beta = 0$ is now possible. The latter modification allows for considering an uninformative gamma prior possibility function when $\alpha = \beta = 0$. The gamma possibility function, like its probabilistic counterpart is a conjugate prior for the precision of a normal probability distribution. The inverse-gamma possibility function with shape parameter $\alpha \geq 0$ and scale parameter $\beta \geq 0$ is simply defined as $\overline{\mathrm{IG}}(\theta; \alpha, \beta) = \overline{\mathrm{G}}(1/\theta; \alpha, \beta)$ due to the simple form of the change of variable formula \eqref{eq:changeOfVariable}. Similarly, the beta possibility function is defined as
$$
\overline{\mathrm{B}}(\theta; \alpha, \beta) = \dfrac{(\alpha+\beta)^{\alpha+\beta}}{\alpha^{\alpha} \beta^{\beta}} \theta^{\alpha} (1-\theta)^{\beta},
$$
for any $\alpha, \beta \geq 0$. Once again the parameters have been shifted by $-1$ when compared to the usual parametrisation of the beta density function. This modification makes the value $\alpha = \beta = 0$ coincide with both the uninformative case and the interpretation of $\alpha$ successes and $\beta$ failures a priori.

\paragraph{Introducing possibility functions.} 

It is generally easy to introduce new possibility functions since the assumption that the supremum is equal to $1$ is much easier to verify than the same assumption with an integral. For instance, any function of the form
$$
\theta \mapsto \exp\bigg( -\Big\| \dfrac{\theta - \mu}{\beta} \Big\|^r\bigg),
$$
with scale parameter $\beta > 0$ and location parameter $\mu$ is a possibility function for any exponent $r > 0$ and any norm $\| \cdot \|$. Renormalising a probability distribution often provides a meaningful candidate for a possibility function, though it is not always the case when the distribution is not a conjugate prior.  A mechanism leading to the definition of a new probability distribution of interest can also inspire the introduction of a new possibility function. Applying this principle to the (non-central) chi-squared distribution leads to the computation of the possibility function describing the sum of squared normally-described uncertain variables as in the following proposition.

\begin{prop}
\label{res:chiSquared}
Let $\uvtheta_1, \dots, \uvtheta_n$ be a collection of uncertain variables independently described by the respective possibility functions $\overline{\mathrm{N}}(\mu_i, \sigma^2)$, $i \in \{1,\dots,n\}$, for some parameters $\mu_i$ and $\sigma^2 > 0$. Then the uncertain variable $\uvpsi = \sum_{i=1}^n \uvtheta^2_i$ is described by the possibility function
\begin{equation}
\label{eq:chiSquaredPossibility}
\overline{\chi}^2(\psi; \bar{\mu},\beta) = \exp\Big(-\dfrac{1}{\beta} (\sqrt{\psi} - \sqrt{\bar{\mu}} )^2 \Big), \qquad \psi \geq 0,
\end{equation}
where $\bar{\mu} = \sum_{i=1}^n \mu_i^2$ and $\beta = 2\sigma^2$.
\end{prop}

The proof of Proposition~\ref{res:chiSquared} can be found in the appendix, together with the proofs of all the other results in this article. As opposed to the chi-squared distribution, $\overline{\chi}^2(\bar{\mu},\beta)$ does not depend directly on the number of terms $n$ in the considered sum of squared normal uncertain variables. The analogue of the standard chi-squared distribution can be recovered when $\bar{\mu} = 0$ and yields the possibility function $\theta \in [0,\infty) \mapsto \exp(-\theta/\beta)$ which can also be identified as the renormalised version of the exponential distribution. It is easy to prove that if $\uvtheta$ is an uncertain variable described by $\overline{\chi}^2(\mu, \beta)$ for some parameters $\mu$ and $\beta > 0$ then, for any constant $\alpha > 0$, the uncertain variable $\alpha \uvtheta$ is described by $\overline{\chi}^2(\alpha\mu, \alpha\beta)$.

\paragraph{Sufficient statistics.}

The notion of sufficient statistics can be defined for uncertain variables as follows: a statistics $\uvt = T(\uvy)$ is sufficient for a model with observation $\uvy$, parameter $\uvtheta$ and likelihood $f_{\uvy|\uvtheta}(\cdot \given \theta)$ if it holds that $f_{\uvy|\uvt,\uvtheta}(y \given \uvt, \uvtheta) = f_{\uvy|\uvt}(y \given \uvt)$, i.e.\ the observation $\uvy$ is conditionally independent of the parameter $\uvtheta$ given $\uvt$. The associated factorisation is
\begin{equation}
\label{eq:sufficientStatistics}
f_{\uvy|\uvtheta}(y \given \theta) = f_{\uvy|\uvt}(y \given t) f_{\uvt|\uvtheta}(t \given \theta), \qquad y \in \mathcal{Y}, t = T(y),
\end{equation}
for some conditional possibility functions $f_{\uvy|\uvt}$ and $f_{\uvt|\uvtheta}$. Indeed, since $\uvt$ is a deterministic transformation of $\uvy$, it holds that $f_{\uvy|\uvtheta}(y \given \theta) = f_{\uvy,\uvt|\uvtheta}(y, T(y) \given \theta)$ which can then be expressed as in \eqref{eq:sufficientStatistics} based on the definition of the notion of sufficient statistics. The same operation is trickier in a probabilistic context since the joint distribution of the observation and the sufficient statistics does not admit a density in general. We emphasise that most statistical tests would be invariant under a change of sufficient statistics in this context since the change of variable formula does not include a Jacobian term and since the decomposition \eqref{eq:sufficientStatistics} involves two possibility functions instead of arbitrary functions as in the probabilistic case.

\begin{example}
\label{ex:normalNormalcase}
We consider observations $\uvy_1,\uvy_2,\dots$ that are independently and identically described by a normal possibility function $f_{\uvy|\uvtheta}(\cdot \given \theta) = \overline{\mathrm{N}}(\theta, \sigma^2)$, conditionally on $\uvtheta = \theta$. The uncertain variable $\uvs_n = n^{-1}\sum_{i=1}^n \uvy_i$ corresponding to the sample average can be referred to as a sufficient statistics since the possibility function $f_{\uvy_{1:n}|\uvtheta}(\cdot \given \theta)$ jointly describing $\uvy_1,\dots,\uvy_n$ given $\uvtheta = \theta$ verifies
\begin{equation*}
f_{\uvy_{1:n}|\uvtheta}(y_{1:n} \given \theta) = \prod_{i=1}^n f_{\uvy|\uvtheta}(y_i \given \theta) = \exp\Big( -\dfrac{1}{2\sigma^2} \sum_{i=1}^n (y_i - s_n)^2 \Big) \exp\Big( -\dfrac{n}{2\sigma^2} (\theta - s_n)^2 \Big),
\end{equation*}
with the first term in the product not depending on $\theta$ and the second only depending on the realisation $s_n$ of the sufficient statistics $\uvs_n$ instead of all observations $y_1,\dots,y_n$. Therefore, if the sole objective is to estimate $\uvtheta$, then marginalising over all sequences of observations $y_{1:n}$ that have sample average $s_n$ makes sense, i.e.\ we consider the likelihood
\begin{align*}
f_{\uvs_n|\uvtheta}( s \given \theta ) & = \sup \bigg\{ f_{\uvy_{1:n}|\uvtheta}(y_{1:n} \given \theta) : y_1,\dots,y_n \in \mathcal{Y}, \frac{1}{n}\sum_{i=1}^n y_i = s \bigg\} \\
& = \exp\Big( -\dfrac{n}{2\sigma^2} (\theta - s)^2 \Big),
\end{align*}
which is indeed a conditional possibility function.
\end{example}

\section{Expected value and variance of uncertain variables}
\label{sec:CLT}

\subsection{LLN and CLT}

In this section we aim to derive a LLN and a CLT for uncertain variables. These results will prove to be useful for our analysis later in the article. This is especially the case in Section~\ref{sec:MLE_un}, where we discuss various asymptotic properties of statistical procedures in the situation where the true sampling distribution is not accessible for analysis. The proofs of the theorems in this section can be found in the appendix.
The more general notations $\uvx$ and $\uvz$ are used for uncertain variables to emphasise that the derived results can be useful in different contexts.

We wish to study sums based on a sequence $\uvx_1,\uvx_2,\dots$ of independently described uncertain variables on $\mathbb{R}^d$ with possibility function $f_{\uvx}$. Related results have been derived by \citet{MM99}, \cite{MM05} and \citet{PT14}, but with different representations of uncertainty or with different notions of expected value. It follows from \eqref{eq:multScalar} that the possibility function $f_{\uvs_n}$ describing the uncertain variable $\uvs_n = n^{-1}\sum_{i=1}^n \uvx_n$ takes the form
$$
f_{\uvs_n}(x) = \sup \bigg\{ \prod_{i=1}^n f_{\uvx}(x_i) : x_1,\dots,x_n \in \mathbb{R}^d,\, \frac{x_1 + \dots + x_n}{n} = x \bigg\},
$$
for any $x \in \mathbb{R}^d$. We can then obtain an analogue of the law of large numbers as follows.

\begin{thm}
\label{thm:LLN}
If $\uvx_1,\uvx_2,\dots$ is a sequence of independent uncertain variables on $\mathbb{R}^d$ with possibility function $f_{\uvx}$ such that 
\begin{enumerate}[label=(\roman*),nosep]
\item \label{it:continuous} $f_{\uvx}$ is continuous on $\mathbb{R}^d$,
\item $f_{\uvx}$ is a twice continuously differentiable function on an open neighbourhood of each point in $\argmax f_{\uvx}$ and
\item \label{it:vanishes} $\displaystyle{\lim_{\norm{x} \to \infty}} f_{\uvx}(x) = 0$, 
\end{enumerate}
then the possibility function $f_{\uvs_n}$ describing the uncertain variable $\uvs_n = n^{-1}\sum_{i=1}^n \uvx_i$ verifies
$$
\lim_{n \to \infty} f_{\uvs_n} = \bm{1}_{\Conv(\argmax f_{\uvx})},
$$
where the limit is point-wise and where $\Conv(S)$ is the convex hull of a set $S \subseteq \mathbb{R}^d$.
\end{thm}

Theorem~\ref{thm:LLN} emphasises the importance of the point(s) where $f_{\uvx}$ is maximised. In the limit, only a number $N(n)$ of uncertain variables in the collection $\uvx_1,\uvx_2,\dots$ can have realisations that are not elements of the $\argmax$ of $f_{\uvx}$, with $N(n)$ verifying $\lim_{n \to \infty} n^{-1}N(n) = 0$. It follows that the sample average is in the convex hull of the $\argmax$. 

Theorem~\ref{thm:LLN} also shows that the $\argmax$ of $f_{\uvx}$, when it is a singleton, plays the role of the expected value in the standard formulation of this result. This suggests a definition of the notion of expectation for uncertain variables. First, a few steps are needed in order to lay some formal basis for such a definition: we consider an uncertain variable $\uvx : \Omega_{\mathrm{u}} \to \mathcal{X}$ and define the o.p.m.\ $\bar{\mathbb{P}}$ on $\Omega_{\mathrm{u}}$ as
\begin{equation*}
\bar{\mathbb{P}}(A) = \sup_{\omega \in A} f_{\uvx}(\uvx(\omega)),
\end{equation*}
for any $A \subseteq \Omega_{\mathrm{u}}$. If the only available information about outcomes and events in $\Omega_{\mathrm{u}}$ comes from the possibility function $f_{\uvx}$ via $\uvx$, then $\bar{\mathbb{P}}$ also contains all that information, which could be used to deduce what is known about any other uncertain variable on $\Omega_{\mathrm{u}}$. The credibility of any event $\uvx \in B$ for some $B \subseteq \mathcal{X}$ can then be measured as $\bar{\mathbb{P}}(\uvx \in B)$ by the standard identification between that event and the subset $\{\omega \in \Omega_{\mathrm{u}} : \uvx(\omega) \in B\}$. In particular, the event $A \subseteq \Omega_{\mathrm{u}}$ is said to happen almost surely under~$\bar{\mathbb{P}}$ if it holds that $\bar{\mathbb{P}}(A^{\mathrm{c}}) = 0$. If $A$ is a singleton $\{\omega_{\mathrm{u}}\}$ then $\bar{\mathbb{P}}(A)$ is simply written $\bar{\mathbb{P}}(\omega_{\mathrm{u}})$. We can now introduce a notion of expectation for uncertain variables as follows.

\begin{defn}
\label{def:expectation}
Given an o.p.m.\ $\bar{\mathbb{P}}$ on a set $\Omega_{\mathrm{u}}$, the expectation of an uncertain variable $\uvx : \Omega_{\mathrm{u}} \to \mathcal{X}$ is defined as\footnote{The notation $\mathbb{E}^*$ is sometimes used to refer to the concept of \emph{outer expectation} which is unrelated to the definition considered here.}
\begin{equation}
\label{eq:expectation}
\mathbb{E}^*(\uvx) = \uvx\Big( \argmax_{\omega \in \Omega_{\mathrm{u}}} \bar{\mathbb{P}}(\omega) \Big).
\end{equation}
\end{defn}

This definition of the expectation does not require assumptions on the space on which the uncertain variable is defined; however, it is important to note that $\mathbb{E}^*(\uvx)$ is set-valued in general. It is consistent with the law of large numbers since
\begin{equation*}
\argmax_{x \in \mathcal{X}} f_{\uvx}(x) = \argmax_{x \in \mathcal{X}} \bar{\mathbb{P}}\big(\uvx^{-1}[\{x\}]\big) = \mathbb{E}^*(\uvx),
\end{equation*}
with $f_{\uvx}$ the possibility function describing $\uvx$; which follows from Proposition~\ref{res:surjectiveTransformArgmax}.

This notion of expectation also displays some useful properties. For instance, if $T$ is a map on $\mathcal{X}$, then it follows from Definition~\ref{def:expectation} that $\mathbb{E}^*(T(\uvx)) = T\big( \mathbb{E}^*(\uvx) \big)$. Unsurprisingly, this invariance property is shared with the MLE which is also based on an $\argmax$. If it holds that $\uvx = (\uvx_1,\uvx_2)$ with $\uvx_1$ and $\uvx_2$ two uncertain variables, then it is easy to prove that $\mathbb{E}^*(\uvx) = (\mathbb{E}^*(\uvx_1), \mathbb{E}^*(\uvx_2))$. By considering the function $T(x_1,x_2) = \alpha x_1 + \beta x_2$ for some $\alpha,\beta \in \mathbb{R}$, it follows that
\begin{equation*}
\mathbb{E}^*(\alpha \uvx_1 + \beta \uvx_2) = \alpha \mathbb{E}^*(\uvx_1) + \beta \mathbb{E}^*(\uvx_2),
\end{equation*}
that is, $\mathbb{E}^*(\cdot)$ is linear. Conditional expectations can also be defined easily, see the appendix for more details.

Given the result of Theorem~\ref{thm:LLN}, it is natural to also consider the $\argmax$ as the starting point of the CLT for uncertain variables. Before stating the theorem, we introduce a slightly more general way of defining the o.p.m.\ $\bar{\mathbb{P}}$ underlying a sequence $\bar{\uvx} = (\uvx_1,\uvx_2,\dots)$ of uncertain variables on $\Omega_{\mathrm{u}}$ as\footnote{A more formal definition would require the introduction of cylinder sets and the corresponding finite-dimensional o.p.m.s.}
\begin{equation*}
\bar{\mathbb{P}}(A) = \sup_{\omega \in A} f_{\bar{\uvx}}(\uvx_1(\omega), \uvx_2(\omega), \dots),
\end{equation*}
for any $A \subseteq \Omega_{\mathrm{u}}$, where $f_{\bar{\uvx}}$ is the possibility function describing all the uncertain variables in the sequence $\bar{\uvx}$ jointly. Henceforth, when some collection of uncertain variables will be defined, the o.p.m.\ $\bar{\mathbb{P}}$ and the corresponding expectation $\mathbb{E}^*(\cdot)$ will be implicitly assumed to be induced by the possibility function jointly describing that collection. This situation also occurs in the Bayesian interpretation of probability where probability distributions represent the state of knowledge rather than some intrinsically random phenomenon whose distribution is induced by the underlying probability space.

\begin{thm}
\label{thm:CLT}
If $\uvx, \uvx_1,\uvx_2,\dots$ is a sequence of uncertain variables on $\mathbb{R}$ independently and identically described by a possibility function $f_{\uvx}$ verifying
\begin{enumerate}[label=(\roman*), nosep]
\item \label{it:logConcave} $f_{\uvx}$ is strictly log-concave and
\item \label{it:diffAtMode} $f_{\uvx}$ is twice differentiable,
\end{enumerate}
then the expected value $\mathbb{E}^*(\uvx)$ of the uncertain variable $\uvx$ is a singleton $\{\mu\}$ and the possibility function $f_{\uvt_n}$ describing the uncertain variable $\uvt_n = n^{-1/2}\sum_{i=1}^n (\uvx_i-\mu)$ verifies
$$
\lim_{n \to \infty} f_{\uvt_n}(t) =
\begin{cases*}
\overline{\mathrm{N}}\big(t; 0, \sigma^2 \big) & if $f_{\uvx}''(\mu) \neq 0$ \\
1 & otherwise,
\end{cases*}
$$
for any $t \in \mathbb{R}$, with $1/\sigma^2 = -f_{\uvx}''(\mu)$.
\end{thm}

Theorem~\ref{thm:CLT} shows that there are two limiting possibility functions instead of a single one as in the standard formulation. Which limiting behaviour applies to $f_{\uvt_n}$ depends on how quickly $f_{\uvx}$ decreases around its $\argmax$. Theorem~\ref{thm:CLT} confirms that normal possibility functions also play a special role in the considered framework.

In the same way that the law of large numbers hinted at the definition of the expectation, 
the form of the asymptotic variance in the CLT suggests a definition of the variance based on the second derivative of $f_{\uvx}$ as
\begin{equation}
\label{eq:variance}
\mathbb{V}^*(\uvx) = - f_{\uvx}''\big(\mathbb{E}^*(\uvx)\big)^{-1}.
\end{equation}
In the case where $f_{\uvx}''(\mathbb{E}^*(\uvx)) = 0$, the limiting possibility function is $\bm{1}_{\mathbb{R}}$ which is equal to the normal possibility function $\overline{\mathrm{N}}(0, \tau^{-1})$ when the precision $\tau$ is set to $0$. It is therefore natural to set $\mathbb{V}^*(\uvx) = \infty$ in this situation. If $f_{\uvx}$ is not twice differentiable but is instead twice left- and right-differentiable at $\mathbb{E}^*(\uvx)$ then left/right variances can be introduced. We only consider the special case where both the left and right derivatives of $f_{\uvx}$ are different from zero for which we define $\mathbb{V}^*(\uvx) = 0$. Overall, the concepts of expected value and variance provide grounds for the Laplace approximation.

There is a natural link between the concept of variance for uncertain variables and the notion of Fisher information. Indeed, since any suitably smooth possibility function $f_{\uvx}$ describing an uncertain variable $\uvx$ verifies $f_{\uvx}''(\mathbb{E}^*(\uvx)) = (\log \circ f_{\uvx})''( \mathbb{E}^*(\uvx) )$, it follows from the properties of the expectation that the variance of the uncertain variable $\uvx$ can be expressed as the inverse of a natural analogue of the notion of Fisher information as 
$$
\mathbb{V}^*(\uvx) = \mathbb{E}^*\bigg(-\frac{\mathrm{d}^2}{\mathrm{d}x^2}\log f_{\uvx}(\uvx)\bigg)^{-1}.
$$
If $\uvx$ is related to a parameter of interest $\theta_0$ via a loss function $\mathcal{L}(\theta_0, \uvx)$ and if we consider that $f_{\uvx}(x) = \exp(-\mathcal{L}(\theta_0, x))$ then $\mathbb{V}^*(\uvx)^{-1} = \mathbb{E}^*(\partial_x^2 \mathcal{L}(\theta_0, \uvx)) = \partial_x^2 \mathcal{L}(\theta_0, x_0)$, with $x_0 = \argmin_{x \in \mathcal{X}} \mathcal{L}(\theta_0, x)$.

In order to illustrate the type of values that $\mathbb{E}^*(\cdot)$ and $\mathbb{V}^*(\cdot)$ take in practice, we consider the different possibility functions introduced so far: if $\uvx$ is described by
\begin{enumerate}[label=(\roman*),nosep]
\item $\overline{\mathrm{N}}(\mu,\sigma^2)$, with $\mu \in \mathbb{R}$ and $\sigma > 0$, then $\mathbb{E}^*(\uvx) = \mu$ and $\mathbb{V}^*(\uvx) = \sigma^2$
\item $\overline{\mathrm{G}}(\alpha, \beta)$, with $\alpha, \beta \geq 0$, then $\mathbb{E}^*(\uvx) = \alpha/\beta$ and $\mathbb{V}^*(\uvx) = \alpha/\beta^2$
\item $\overline{\mathrm{B}}(\alpha, \beta)$, with $\alpha, \beta \geq 0$, then $\mathbb{E}^*(\uvx) = \alpha / (\alpha + \beta)$ and $\mathbb{V}^*(\uvx) = \alpha\beta / (\alpha+\beta)^3$
\end{enumerate}
These examples show that the expected value (and possibly the variance) can match between a probability distribution and the corresponding possibility function. Yet, this is not necessarily the case: if $\uvx$ is described by $\overline{\mathrm{G}}(\alpha, \beta)$ then $1/\uvx$ is described by the inverse-gamma possibility function $\overline{\mathrm{IG}}(\alpha, \beta)$ and $\mathbb{E}^*(1/\uvx) = 1/\mathbb{E}^*(\uvx) = \beta/\alpha$ which differs from the mean of an inverse-gamma random variable with the same parameters. Similarly, if $\uvx$ is described by the chi-squared possibility function $\overline{\chi}^2(\mu,\beta)$ for some parameter $\mu$ and $\beta>0$ then $\mathbb{E}^*(\uvx) = \mu$ and $\mathbb{V}^*(\uvx) = 0$ since this possiblity function has non-zero left and right derivatives at $\mu$. The expected value and variance can also help identify a natural parametrisation for a possibility function as illustrated in the following example.

\begin{example}
If two random variables $X$ and $Z$ are such that $Z$ is gamma distributed and the conditional distribution of $X$ given $Z=z$ is $\mathrm{N}(\mu, z^{-1})$ for some $\mu \in \mathbb{R}$, then the marginal distribution of $X$ is a generalised Student's t distribution. One can therefore consider two uncertain variables $\uvx$ and $\uvz$ such that $\uvx$ is described conditionally by $f_{\uvx|\uvz}(x \given z) = \overline{\mathrm{N}}(x; \mu, z^{-1})$ and $f_{\uvz}(z) = \overline{\mathrm{G}}(z; \alpha, \beta)$ for some $\alpha \geq 0$ and $\beta \geq 0$. In this situation, the parameter $\uvx$ is marginally described by 
\begin{equation}
\label{eq:marginalisedNormalGamma}
f_{\uvx}(x) = \bigg( 1 + \dfrac{1}{2\beta}(\mu - x)^2 \bigg)^{-\alpha}.
\end{equation}
Since the expected value and variance are $\mathbb{E}^*(\uvx) = \mu$ and $\mathbb{V}^*(\uvx) = \beta/\alpha$, it reasonable to define the generalised Student's t possibility function as
$$
\overline{\mathrm{St}}(x; \nu, \mu, s) = \bigg( 1 + \dfrac{(\mu - x)^2}{\nu s} \bigg)^{-\nu/2},
$$
with $\mu \in \mathbb{R}$, $\nu \geq 0$ and $s > 0$ which verifies $\mathbb{E}^*(\uvx) = \mu$ and $\mathbb{V}^*(\uvx) = s$. 
\end{example}

The LLN and CLT for uncertain variables are expressed via point-wise convergence of possibility functions, which can be identified as a mode of convergence for uncertain variables: if we consider a sequence of uncertain variables $(\uvx_1, \uvx_2, \dots)$ on some state space $\mathcal{X}$ then we say that this sequence converges in \emph{outer probability measure} to another uncertain variable $\uvx$ if
\begin{equation}
\label{eq:convOpm}
\lim_{n \to \infty} \bar{\mathbb{P}} (\uvx_n = x) = \bar{\mathbb{P}} \big( \uvx = x \big),
\end{equation}
for any $x \in \mathcal{X}$. This notion of convergence is denoted $\uvx_n \xrightarrow{o.p.m.} \uvx$ or $\uvx_n \xrightarrow{o.p.m.} f_{\uvx}$ where $f_{\uvx}$ is the possibility function describing $\uvx$. This is equivalent to a point-wise convergence for the corresponding sequence of possibility functions and is therefore the convergence given by the CLT for uncertain variables. Note that this notion of convergence only makes sense for o.p.m.s defined as the supremum of a possibility function.

\subsection{Inference with o.p.m.s}

We consider that the received observations $y_1, y_2, \dots$ are realisations of i.i.d.\ random variables $Y_1, Y_2, \dots$ with unknown common distribution $p_Y$. The notion of expected value introduced in Definition~\ref{def:expectation}  suggests that, whenever the prior uncertainty is modelled with a possibility function, the MAP is the most natural point estimator. Although the MAP is also important when representing the posterior as a probability distribution, there are some disadvantages, e.g., the (probabilistic) MAP might not transform coherently under reparametrisation \citep{DM07}. This is not the case with possibility functions since it holds that $\mathbb{E}^*(T(\uvtheta) \given y_{1:n}) = T\big(\mathbb{E}^*(\uvtheta \given y_{1:n}) \big)$ for any parameter $\uvtheta$ of interest, any sequence $y_{1:n}$ of observations and any mapping $T$.

\subsubsection{With the likelihood as a probability distribution}

We first consider the case where the likelihood is based on a family of probability distributions $\{p_{Y|\uvtheta}(\cdot \given \theta)\}_{\theta \in \Theta}$. There are several interesting connections between the Bayesian nature of $f_{\uvtheta|Y}(\cdot \given y)$ as defined in \eqref{eq:likelihoodAsProbability} and frequentist inference when the prior $f_{\uvtheta}$ is assumed to be uninformative, i.e.\ $f_{\uvtheta} = \bm{1}_{\Theta}$; it indeed holds that
\begin{enumerate}[label=(\roman*), nosep]
\item The posterior expected value $\mathbb{E}^*(\uvtheta \given Y = y)$ is the MLE of the parameter, which we denote by $\hat{\theta}(y)$.
\item The function $f_{\uvtheta|Y}(\theta_0 \given \cdot)$ for a given point $\theta_0 \in \Theta$ is the (simple) LRT for the hypothesis $\mathrm{H}_0: \uvtheta = \theta_0$ versus $\mathrm{H}_1: \uvtheta \neq \theta_0$.
\item The inverse of the posterior variance of $\uvtheta$ is
$$
\mathbb{V}^*(\uvtheta \given Y = y)^{-1} = \mathbb{E}^*\bigg(-\dfrac{\partial^2}{\partial \theta^2} \log p_{Y|\uvtheta}(y \given \uvtheta) \,\bigg|\, Y = y \bigg) = -\dfrac{\partial^2}{\partial \theta^2} \log p_{Y|\uvtheta}(y \given \hat\theta(y)),
$$
which is the observed information, denoted by $\mathcal{J}^*$.
\end{enumerate}
These aspects reinforce the fact that possibility functions represent information in a way that is compatible with frequentist inference. Before continuing with our analysis, we introduce in the next section a few additional concepts for the case where the likelihood is a possibility function.

\subsubsection{With the likelihood as a possibility function}

We follow the approach of generalised Bayesian inference and consider the situation where it is not convenient to model the true data-generating process. Instead, in order to learn some characteristics of this process, we identify a likelihood (as a possibility function) for which the MLE tends to the quantity of interest as the number of observations tends to infinity. A posterior representation of the uncertainty can then be determined, potentially using some prior information about the parameter of interest if any is available. In this case, the uncertainty about the relation between the observations $y_1,y_2,\dots$ and the parameter of interest $\uvtheta$ can also be modelled via a sequence of independently-described uncertain variable $\uvy_1,\uvy_2,\dots$ with common possibility function $f_{\uvy|\uvtheta}(\cdot \given \theta)$, conditionally on $\uvtheta = \theta$. In this context, the MAP is an uncertain variable defined as $\uvtheta^*_n = \mathbb{E}^*(\uvtheta \given \uvy_{1:n})$ and depends on the choice of prior possibility function $f_{\uvtheta}$. If $f_{\uvtheta} = \bm{1}_{\Theta}$ then $\uvtheta^*_n$ is the MLE, denoted by $\hat{\uvtheta}_n$ and defined as
$$
\hat{\uvtheta}_n = \argmax_{\theta \in \Theta} \dfrac{1}{n} \sum_{i=1}^n \log f_{\uvy|\uvtheta}(y_i \given \theta).
$$

Since none of the possibility functions in the parametric family $\{f_{\uvy|\uvtheta}(\cdot \given \theta)\}_{\theta \in \Theta}$ corresponds to the true sampling distribution of the observations, we must find another way to define the best-fitting parameter $\theta_0$ which we will refer to as the \emph{true} parameter for simplicity. One possible approach is to follow the M-estimation literature \citep{VPG91} and use consistency as a way to define $\theta_0$ rather than as a result about the convergence of the considered estimator to the actual true parameter as is standard. This approach is also the one used by \citet{BHW16}. Indeed, it follows from the standard LLN that
\begin{equation}
\label{eq:MLE_limit}
\dfrac{1}{n} \sum_{i=1}^n \log f_{\uvy|\uvtheta}(Y_i \given  \theta) \xrightarrow{p.} \mathbb{E}( \log f_{\uvy|\uvtheta}(Y \given \theta) ),
\end{equation}
where $\xrightarrow{p.}$ and $\mathbb{E}(\cdot)$ denote the convergence in probability and the expectation w.r.t.\ $p_Y$ respectively. The true parameter $\theta_0$ is then defined as the point where the right hand side of \eqref{eq:MLE_limit} is maximized. If there is no such point then $\theta_0$ is not well-defined. The standard concept of identifiability translated for possibility functions, that is $f_{\uvy|\uvtheta}(\cdot \given \theta) \neq f_{\uvy|\uvtheta}(\cdot \given \psi)$ for any $\theta, \psi \in \Theta$ such that $\theta \neq \psi$, is not sufficient to ensure the existence of such a point $\theta_0$. We instead consider the following stronger assumption: $f_{\uvy|\uvtheta}(y \given \cdot)$ is strictly log-concave for any $y \in \mathcal{Y}$. Indeed, it follows from this assumption that $\mathbb{E}( \log f_{\uvy|\uvtheta}(Y \given \cdot))$ is strictly concave and has a single maximizer $\theta_0$. This assumption is in line with the other ones made throughout the article, especially in the CLT for uncertain variables. In the remainder of this section, we will assume that $\uvy$ is described by $f_{\uvy|\uvtheta}(\cdot \given \theta_0)$ and we will denote by $\mu_0$ the expected value of $\uvy$.

\begin{remark}
Since any estimator $\uvtheta_n$ of $\theta_0$ is itself an uncertain variable, the usual characteristics of estimators have to be redefined. For instance, the bias of $\uvtheta_n$ can be defined as $b^*(\uvtheta_n) = \mathbb{E}^*( \uvtheta_n - \theta_0 \given \theta_0)$ and, naturally, one can refer to $\uvtheta_n$ as an unbiased estimator if $b^*(\uvtheta_n) = 0$.
\end{remark}

Having used the standard notion of consistency to define $\theta_0$, we can now introduce a form of consistency for estimators based on uncertain variables. However, before introducing this weaker notion of consistency, we introduce a suitable version of the concept of identifiability as follows.

\begin{defn}
Let $\{f_{\uvy|\uvtheta}(\cdot\given \theta)\}_{\theta \in \Theta}$ be a parametrized family of conditional possibility functions on $\mathcal{Y}$ describing an uncertain variable $\uvy$ and let $\theta_0 \in \Theta$ be the true parameter. The parameter $\theta_0$ is said to be identifiable if, for any $\theta \in \Theta$, $\mathbb{E}^*(\uvy\given \theta)$ is a singleton and $\mathbb{E}^*(\uvy\given \theta) = \mathbb{E}^*(\uvy \given \theta_0)$ implies $\theta = \theta_0$.
\end{defn}

This is a much stronger assumption than standard identifiability since any parameter that does not affect the mode of the corresponding likelihood would be unidentifiable. Location parameters are typically identifiable but they are not the only ones; for instance, if the likelihood is a gamma possibility function, then each of the two parameters is identifiable when assuming the other one is fixed. When considering loss functions, the identifiability assumption implies that the loss $\mathcal{L}(\theta_0, \cdot)$ should be minimised at a point where no other loss function $\mathcal{L}(\theta, \cdot)$, $\theta \neq \theta_0$, is minimised.

\begin{example}
If the likelihood $f_{\uvy|\uvtheta}(\cdot \given \theta)$ is equal to $\overline{\mathrm{N}}(\theta, \sigma^2)$ for some $\sigma > 0$ then any $\theta \in \mathbb{R}$ is identifiable. However, the likelihood $\overline{\mathrm{N}}(\mu, \theta)$, $\theta > 0$, is nowhere identifiable; indeed, the expected value $\mathbb{E}^*(\uvy \given \theta) = \mu$ is independent of $\theta$. Therefore, one must use a p.d.f.\ as a likelihood when trying to learn the variance of a sequence of observations.
\end{example}

Some weak form of consistency can be obtained for the MLE $\hat{\uvtheta}_n$ by assuming that the parameter $\theta_0$ is identifiable. Indeed, the LLN for uncertain variables yields that
$$
\dfrac{1}{n} \sum_{i=1}^n \log f_{\uvy|\uvtheta}(\uvy_i \given \theta) \xrightarrow{o.p.m.} \mathbb{E}^*\big( \log f_{\uvy|\uvtheta}(\uvy \given \theta)\big) = \log f_{\uvy|\uvtheta}(\mu_0 \given \theta).
$$
The identifiability assumption ensures that only $f_{\uvy|\uvtheta}(\cdot\given\theta_0)$ has $\mu_0$ as a mode, so that $\theta_0$ is the unique maximizer of $\log f_{\uvy|\uvtheta}(\mu_0 \given \cdot)$, and it holds that $\mathbb{E}^*(\log f_{\uvy|\uvtheta}(\uvy \given \theta_0)) = 0$ as in the situation where a probabilistic model is correctly specified.

A practical illustration of the use of the LLN and CLT is given in the next example. 

\begin{example}
\label{ex:LLN}
We continue with the situation introduced in Example~\ref{ex:normalNormalcase} where $f_{\uvy|\uvtheta}(\cdot \given \uvtheta)$ is equal to $\overline{\mathrm{N}}(\uvtheta, \sigma^2)$ and where the sample average $\uvs_n = n^{-1}\sum_{i=1}^n \uvy_i$ is a sufficient statistics. We now consider that the observations come from a sequence of i.i.d.\ random variables $Y,Y_1,Y_2,\dots$ with unknown common law $p_Y$. The standard law of large numbers yields that the sample average $S_n = n^{-1}\sum_{i=1}^n Y_i$ tends to $\mathbb{E}(Y)$ whereas Theorem~\ref{thm:LLN} yields that $\uvs_n$ tends to $\mathbb{E}^*(\uvy \given \theta_0) = \theta_0$, with $\theta_0$ the true value of the parameter. Combining these two results, we obtain that $\theta_0 = \mathbb{E}(Y)$, i.e., our model based on possibility functions targets $\mathbb{E}(Y)$ although we did not provide a probabilistic model for $p_Y$. This type of result is usual in M-estimation, yet, we highlight that the proposed approach additionally provides a posterior quantification of the uncertainty via $f_{\uvtheta|\uvy_{1:n}}(\cdot \given y_{1:n})$ which can take into account a possibility function $f_{\uvtheta}$ describing $\uvtheta$ a priori. Under the considered modelling assumptions, the sequence of possibility functions describing $\sqrt{n}(\uvs_n - \theta_0)$ converges point-wise to $\overline{\mathrm{N}}(0, \sigma^2)$.
\end{example}

Example~\ref{ex:LLN} gives a result of asymptotic normality for estimating the expected value with the likelihood defined as a normal possibility function. More general asymptotic normality results will be given in further sections.

\section{Asymptotic analysis with unknown sampling distribution}
\label{sec:MLE_un}

In this section we analyse various asymptotic properties of statistical procedures in the situation where the true sampling distribution is not available for analysis. The likelihood of a parameter $\theta \in \Theta \subseteq \mathbb{R}$ for the observations $y_1,\dots,y_n$ in $\mathcal{Y} \subseteq \mathbb{R}$ can be either a probability distribution or a possibility function so we denote it by $L_n(\theta)$ in general. The corresponding log-likelihood is denoted by $\ell_n(\theta)$. We model the uncertainty about the true parameter $\theta_0$ via an uncertain variable $\uvtheta$ on $\Theta$ and we consider that there is some prior information about $\uvtheta$ taking the form of a prior possibility function $f_{\uvtheta}$. We formulate the following assumptions:
\begin{enumerate}[label=A.\arabic*, nosep, series=assump]
\item \label{it:identifiable} The parameter $\theta_0$ is well defined and the MLE is consistent
\item \label{it:priorDiff} The prior possibility function $f_{\uvtheta}$ is continuous in a neighbourhood of $\theta_0$
\end{enumerate}
The generality of Assumption~\ref{it:identifiable} makes it suitable for the two types of likelihood that are considered. The specific conditions under which Assumption~\ref{it:identifiable} holds for a probabilistic likelihood are the usual ones, see for instance \cite{Aw49} and \cite{LLC90}.

\subsection{Bayesian inference} 
\label{sec:BvM}

One of the most important asymptotic results in Bayesian statistics is the BvM theorem which can be interpreted as follows: given some initial distribution, the posterior becomes independent of the prior as we take the number of observations to infinity, and the posterior tends to a normal distribution. The posterior possibility function describing $\uvtheta$ given a sequence of observations $(y_1,\dots,y_n)$ can be expressed as
$$
f_{\uvtheta}(\theta \given  y_{1:n}) = \dfrac{L_n(\theta) f_{\uvtheta}(\theta)}{\sup_{\theta' \in \Theta} L_n(\theta') f_{\uvtheta}(\theta')}.
$$
We omit the variables on which the possibility function describing $\uvtheta$ is conditioned on in the posterior $f_{\uvtheta}(\cdot \given  y_{1:n})$ since they could be either random variables or uncertain variables. In this context, the BvM theorem takes the following form.

\begin{thm}[Bernstein-von Mises]
\label{res:BvM}
Under Assumptions~\ref{it:identifiable}-\ref{it:priorDiff}, it holds that, for large values of $n$,
$$
f_{\uvtheta}(\theta \given y_{1:n}) \approx \overline{\mathrm{N}}\bigg(\theta; \theta_0 + \dfrac{\Delta_n}{\sqrt{n}}, \dfrac{1}{\mathcal{J}^*_n} \bigg),
$$
where $\Delta_n = \sqrt{n} \partial_{\theta} \ell_n(\theta_0) / \mathcal{J}^*_n$ and where
\begin{equation}
\label{eq:obsInfo}
\mathcal{J}^*_n = - \partial_{\theta}^2 \ell_n(\hat{\theta}_n)
\end{equation}
is the observed information with $\hat{\theta}_n$ the MLE based on $y_{1:n}$.
\end{thm}

Theorem~\ref{res:BvM} states that the posterior possibility function is approximately normal with expected value $\theta_0 + \Delta_n/\sqrt{n}$ (which tends to $\theta_0$ as $n$ tends to infinity) and with the inverse of $\mathcal{J}^*_n$ as variance. An important aspect of the proof of the BvM theorem is that the nature of the likelihood $L_n$ (expressed as a probability distribution or as a possibility function) had no bearing on the calculations. Although the BvM theorem guarantees (under assumptions) that the prior information is forgotten in the limit, the nature of the prior information defines the nature of posterior representation of uncertainty as a probability distribution or as a possibility function. In the next section, we consider how we can predict the behaviour of the MAP $\mathbb{E}^*(\uvtheta \given y_{1:n})$ and of the credibility of the true parameter value $f_{\uvtheta}(\theta_0 \given y_{1:n})$ when varying the observations. When the prior $f_{\uvtheta}$ is uninformative, these quantities are respectively the (generalised) MLE and likelihood ratio test as highlighted earlier.

\subsection{Frequentist-type guarantees}

The main reason for the weak dependency of the BvM theorem on the nature of the likelihood is that the observations are fixed in a Bayesian setting. However, in order to derive frequentist-type guarantees for the the MAP $\mathbb{E}^*(\uvtheta \given y_{1:n})$ and for the credibility of the true parameter value $f_{\uvtheta}(\theta_0 \given y_{1:n})$, the variability of the sequence of observations must be taken into account. We henceforth consider the case where the likelihood is a possibility function and make explicit the dependency of the likelihood on the data by writing $L(\cdot\,; y_{1:n})$ instead of $L_n$ and similarly for the log-likelihood. In order to model the fact that the data-generating process is not known, we define $\uvy$ as the uncertain variable described by the possibility function $f_{\uvy} = f_{\uvy|\uvtheta}(\cdot \given \theta_0)$ where $\theta_0$ is the true value of the parameter.

The observations $y_1,\dots,y_n$ are assumed to be realisations of a sequence of independent uncertain variables $\uvy_1,\dots,\uvy_n$ identically described by $f_{\uvy}$, and we denote by $\bar{\mathbb{P}}$ the induced o.p.m.\ on $\Omega_{\mathrm{u}}$. Since the observations are now uncertain variables, the MAP itself is an uncertain variable defined as $\uvtheta^*_n = \mathbb{E}^*(\uvtheta \given \uvy_{1:n})$.

Before deriving asymptotic normality results, we need to introduce the analogues of the probabilistic results used when performing a standard statistical analysis. A concept that is needed to state the assumptions and prove asymptotic normality is the analogue of the usual dominance property for sequences of random variables, which can be rephrased as: the sequence of uncertain variables $(\uvx_1, \uvx_2, \dots)$ on $\mathcal{X}$ is \emph{dominated by another sequence $(s_1, s_2, \dots)$ on $\mathcal{X}$}, or is $O(s_n)$, if for all $\epsilon > 0$ there exist $N_{\epsilon} > 0$, $\delta_{\epsilon} > 0$ such that $\bar{\mathbb{P}} \big( \| \uvx_n \|/\| s_n \| > \delta_{\epsilon} \big) < \epsilon$ for all $n > N_{\epsilon}$, where $\bar{\mathbb{P}}$ is induced by the sequence $(\uvx_1, \uvx_2, \dots)$. We can now formulate the assumptions required for the asymptotic analysis of this section: 
\begin{enumerate}[resume*=assump]
\item \label{it:logConcave2} $f_{\uvy}$ is strictly log-concave
\item The mapping $(\theta,y) \mapsto L(\theta; y)$ is twice-differentiable in $y$ as well as thrice-differentiable in $\theta$
\item \label{it:bounded2} $\partial_{\theta}^3 \ell(\theta; \uvy_{1:n}) = O(n)$ under the o.p.m.\ $\bar{\mathbb{P}}$
\end{enumerate}
Although Assumption~\ref{it:logConcave2} is restrictive, it is always possible to introduce a tempered version $f^{\mathrm{t}}_{\uvy}$ of the possibility function $f_{\uvy}$, with $f^{\mathrm{t}}_{\uvy}(y) \geq f_{\uvy}(y)$ for any $y \in \mathcal{Y}$, which satisfies it. The cost of such a change of possibility function will induce an increase in asymptotic variance since the variance of $\uvy$ under the tempered possibility function will be necessarily greater than under $f_{\uvy}$.

In this context, the score function is defined as $s_{\theta_0}(\cdot) = \partial_{\theta} \ell(\theta_0; \cdot)$ and its expected value is equal to $0$; indeed $\mathbb{E}^* ( s_{\theta_0}(\uvy)) = s_{\theta_0}( \mathbb{E}^*(\uvy)) = 0$. However, the Fisher information cannot be defined as the expectation of the squared score as is usual; in fact, the expected value of the latter is also equal to $0$. Instead, we introduce the Fisher information $\mathcal{I}^*(\theta)$ as
\begin{equation}
\label{eq:Fisher}
\mathcal{I}^*(\theta) = \mathbb{E}^*( -\partial^2_{\theta} \ell(\theta; \uvy) \given \theta ).
\end{equation}
The following example illustrates one of the advantages of this notion of Fisher information.

\begin{remark}
\citet[Example~7.2.1]{EL04} highlights some limitations of the standard notion of Fisher information $\mathcal{I}(\theta)$ by pointing out that $\mathcal{I}(\theta) = 0$ for the likelihood $p_Y(y \given \theta) = \mathrm{N}(y; \theta^3, \sigma^2)$ at $\theta = 0$, which is counter-intuitive. With uncertain variables, the variance corresponding to the possibility function $f_{\uvy|\uvtheta}(y \given \theta) = \overline{\mathrm{N}}(y; \theta^3, \sigma^2)$ is infinite when $\theta = 0$, which is consistent with $\mathcal{I}^*(\theta) = 0$.
\end{remark}

\subsubsection{Asymptotic normality}

One of the most important asymptotic results in frequentist statistics is the asymptotic normality of the MLE. We show in the following theorem that a similar result holds for the MAP $\uvtheta^*_n$, which is equal to the MLE when the prior $f_{\uvtheta}$ is uninformative.

\begin{thm}
\label{thm:asympNormUnknownSamp}
Under Assumptions~\ref{it:identifiable}-\ref{it:bounded2}, it holds that $\sqrt{n}(\uvtheta^*_n - \theta_0) \xrightarrow{o.p.m.} \overline{\mathrm{N}}(0, \sigma^2)$ with $\sigma^2 = \mathbb{V}^*(s_{\theta_0}(\uvy))/\mathcal{I}^*(\theta_0)^2$. 
\end{thm}

The advantage of Theorem~\ref{thm:asympNormUnknownSamp} is that it provides some quantitative means of assessing the asymptotic performance of the MAP even when the true data generating process is not known. Indeed, the asymptotic variance in Theorem~\ref{thm:asympNormUnknownSamp} only involves derivatives of the likelihood and no expectations w.r.t.\ the true data distribution.

It follows from Assumption~\ref{it:logConcave2} that the score $s_{\theta_0}$ is continuous and strictly monotone and is therefore invertible. Assuming that both $s_{\theta_0}$ and its inverse are twice differentiable and applying Proposition~\ref{prop:mapVariance}, we can express the term $\mathbb{V}^*(s_{\theta_0}(\uvy))$ in the asymptotic variance of Theorem~\ref{thm:asympNormUnknownSamp} more explicitly as
\begin{equation*}
\mathbb{V}^*(s_{\theta_0}(\uvy)) = \mathbb{V}^*(\uvy) \big( \partial_y s_{\theta_0}(\mu_0) \big)^2,
\end{equation*}
with $\mu_0 = \mathbb{E}^*(\uvy)$. It is interesting to see the second derivative in $y$ appear in the expression of the asymptotic variance (via the variance term $\mathbb{V}^*(\uvy)$) since the latter usually only depends on the curvature of the log-likelihood $\ell(\theta; \cdot)$ as a function of $\theta$ via the Fisher information. The curvature of $\ell(\theta; \cdot)$ evaluated at $\theta_0$ is associated with the variation in the observations at the true parameter value, which might indeed influence the algorithm. In particular, a log-likelihood which would be ``flat'' as a function of $y$ around the point $\mu_0$ would yield a potentially infinite asymptotic variance via the term $\mathbb{V}^*(\uvy)$. Two special cases can then be identified: If $\theta$ is a location parameter then the asymptotic variance in Theorem~\ref{thm:asympNormUnknownSamp} simplifies to $\mathbb{V}^* (\uvy)$ since it holds that $( \partial_y s_{\theta_0}(\mu_0))^2 = \mathcal{I}^*(\theta_0)^2$. If the likelihood is based on an exponentiated reward function $-\mathcal{L}(\theta, \cdot)$ then the asymptotic variance in Theorem~\ref{thm:asympNormUnknownSamp} can be expressed as
$$
\bigg( \dfrac{\partial_y \partial_{\theta} \mathcal{L}(\theta_0,\mu_0)}{\partial_{\theta}^2 \mathcal{L}(\theta_0,\mu_0)} \bigg)^2 \partial_y^2 \mathcal{L}(\theta_0,\mu_0),
$$
with $\mu_0 = \argmin_{y \in \mathcal{Y}} \mathcal{L}(\theta_0, y)$. This expression highlights the structure of the asymptotic variance as a combination of second-order derivatives.

\subsubsection{Hypothesis testing}
\label{sec:LRT_unknown}

We consider a simple hypothesis test $\mathrm{H}_0 : \uvtheta = \theta_0$ versus $\mathrm{H}_1:\uvtheta \neq \theta_0$ of the form $\lambda(y_{1:n}) \doteq f_{\uvtheta}(\theta_0 \given y_{1:n})$, where, as opposed to LRTs, prior information can be taken into account via the prior $f_{\uvtheta}$. The null hypothesis is then rejected if $\lambda(y_{1:n})$ is less than or equal to some threshold $c$ to be determined. We follow the usual approach and recast the problem of setting a value for $c$ into choosing a level of confidence $\alpha \in (0,1)$ and solving the inequality $\bar{\mathbb{P}}( \lambda(\uvy_{1:n}) \leq c) \leq \alpha$. Assuming for simplicity that $f_{\uvtheta}(\cdot \given  y)$ is unimodal, the proposed interpretation facilitates the calculation of an $\alpha$-credible interval $[a,b]$ as
$$
f_{\uvtheta}(a \given  y) = f_{\uvtheta}(b \given  y) = \alpha \qquad \text{and} \qquad a < \mathbb{E}^*(\uvtheta \given  y) < b.
$$
When interpreting the o.p.m.\ $\bar{P}_{\uvtheta}(\cdot \given  y)$, characterised by $\bar{P}_{\uvtheta}(B \given  y) = \sup_{\theta \in B} f_{\uvtheta}(\theta \given  y)$, $B \subseteq \mathbb{R}$, as an upper bound for a subjective probability distribution $p_{\uvtheta}(\cdot \given y)$, we can deduce that $p_{\uvtheta}([a,b] \given y) \geq 1 - \alpha$ from the fact that $\bar{P}_{\uvtheta}([a,b]^{\mathrm{c}} \given y) = \alpha$. As is usual for LRTs, we use the asymptotic properties of $\lambda(\uvy_{1:n})$ to make the computation of $c$ feasible.

\begin{thm}
\label{thm:asymptotic_LRT_unknown}
Under Assumptions~\ref{it:identifiable}-\ref{it:bounded2}, it holds that $-2 \log \lambda(\uvy_{1:n}) \xrightarrow{o.p.m.} \overline{\chi}^2(0, \beta)$ with $\beta = \mathbb{V}^*( s_{\theta_0} (\uvy))/\mathcal{I}^*(\theta_0)$.
\end{thm}

Using Theorem~\ref{thm:asymptotic_LRT_unknown}, it is easy to approximate the credibility of the event $\lambda(\uvy_{1:n}) \leq c$ for large values of $n$ and deduce a threshold $c$ for a given confidence level $\alpha$.

\subsection{Numerical results: ratio of two means}

Let $(Y, Y'), (Y_1,Y'_1), (Y_2,Y'_2), \dots$ be a collection of i.i.d.\ random variables sampled from an unknown distribution with $Y$ and $Y'$ real-valued and independent. Consider the objective of estimating the ratio between $\mathbb{E}(Y)$ and $\mathbb{E}(Y')$. The most natural estimator for such a quantity is $\hat{\theta}_n = \bar{Y}_n / \bar{Y}'_n$ with $\bar{Y}_n = n^{-1} \sum_{i=1}^n Y_i$ and $\bar{Y}'_n = n^{-1} \sum_{i=1}^n Y'_i$ the sample averages. In order to apply the proposed approach, we must identify likelihoods with a suitable MLEs in order to learn both $\mathbb{E}(Y)$ and $\mathbb{E}(Y')$. The normal distribution is a suitable likelihood in this case, however, since the variances of these observations are not known either, they must also be estimated. We therefore consider the likelihoods $\mathrm{N}(\bm{\mu}, \bm{\tau}^{-1})$ and $\mathrm{N}(\bm{\mu}', \bm{\tau}'^{-1})$ where both the means and the precisions are uncertain variables to be determined. We now focus on the prior possibility functions for $\bm{\mu}$ and $\bm{\tau}$ since the ones for $\bm{\mu}'$ and $\bm{\tau}'$ will be of the same form. The prior possibility function $f_{\bm{\tau}}$ describing $\bm{\tau}$ is assumed to be the gamma possibility function $\overline{\mathrm{G}}(\alpha_0, \beta_0)$ for some given $\alpha_0,\beta_0 \geq 0$. The prior possibility function $f_{\bm{\mu}}(\cdot \given \tau)$ for $\bm{\mu}$ given $\bm{\tau} = \tau$ is assumed to be the normal possibility function $\overline{\mathrm{N}}(\mu_0, (k\tau)^{-1})$ for some given $\mu_0 \in \mathbb{R}$ and $k \geq 0$. The corresponding parameters for $\bm{\mu}'$ and $\bm{\tau}'$ are denoted by $\alpha'_0$, $\beta'_0$, $\mu'_0$ and $k'$. The posterior possibility functions describing $\bm{\mu}$ and $\bm{\tau}$ are of the same form as the priors, i.e.\
$$
f_{\bm{\tau}}(\tau \given y_{1:n}) = \overline{\mathrm{G}}(\tau; \alpha_n, \beta_n) \qquad \text{and} \qquad f_{\bm{\mu}}(\mu \given \tau, y_{1:n}) = \overline{\mathrm{N}}(\mu; \mu_n, ((k + n)\tau)^{-1})
$$
with $\alpha_n = \alpha_0 + n/2$, $\mu_n = (k\mu_0 + n \bar{y}_n)/(k + n)$ and
$$
\beta_n = \beta_0 + \dfrac{n\hat{v}_n}{2} + \dfrac{nk}{2(n+k)}(\mu_0 - \bar{y}_n)^2
$$
where $\hat{v}_n = n^{-1}\sum_{i=1}^n (y_i - \bar{y}_n)^2$ is the sample variance. These results are consistent with the probabilistic case as detailed, e.g., in \citep{GCSDVR13}. The marginal possibility function describing $\bm{\mu}$ given $y_{1:n}$ is
$$
f_{\bm{\mu}}(\mu \given y_{1:n}) = \overline{\mathrm{St}}(\mu; \nu_n, \mu_n, s_n),
$$
with $\nu_n = 2\alpha_n$ degrees of freedom and variance $s_n = \beta_n/(\alpha_n(k+n))$. The possibility function describing $\bm{r} \doteq \bm{\mu}/\bm{\mu}'$ given $y_{1:n}$ is
$$
f_{\bm{r}}(r \given y_{1:n}) = \sup \big\{ f_{\bm{\mu}}(\mu \given y_{1:n}) f_{\bm{\mu}'}(\mu' \given y_{1:n}) : \mu/\mu' = r \big\}.
$$
While this expression does not seem to have a simple closed-form expression, it is straightforward to compute numerically using any optimisation algorithm. Using the properties of the expected value $\mathbb{E}^*(\cdot)$, we find that
$$
\mathbb{E}^*( \bm{r} \given y_{1:n}) = \dfrac{\mathbb{E}^*( \bm{\mu} \given y_{1:n})}{\mathbb{E}^*( \bm{\mu'} \given y_{1:n})} = \dfrac{\mu_n}{\mu'_n}.
$$ 
Considering the simplified case where $k = k' = \alpha_0 = \alpha'_0 = \beta_0 = \beta'_0 = 0$, i.e.\ the case where there is no prior information on either of $\bm{\mu}$, $\bm{\mu}'$, $\bm{\tau}$ and $\bm{\tau}'$ (which would be improper for random variables), we recover the estimate $\mathbb{E}^*(\bm{r} \given y_{1:n}) = \bar{y}_n / \bar{y}'_n$. In this case, it holds that $\alpha_n = n/2$ and $\beta_n = n\hat{v}_n/2$, so that the expected value of $\bm{\tau}$ is $\alpha_n/\beta_n = 1/\hat{v}_n$. Because of the properties of the expected value we have $\mathbb{E}^*(\bm{\tau}^{-1}) = \hat{v}_n$. With the probabilistic modelling, the inverse of the expected value of the precision ($\alpha_n/\beta_n = 1/\hat{v}_n$) would not be equal to the expected value of the variance ($\beta_n/(\alpha_n-1) = n\hat{v}_n/(n-2)$). Although the proposed approach does not yield an improved estimator in the case where the prior is uninformative, a quantification of the posterior uncertainty is now available. We also consider the case where the variances of $Y$ and $Y'$ are known, in which case $\bm{r}$ is simply a ratio of normally-described uncertain variables. We refer to this simpler model as the ``normal model'' and use it to show that the results are not specific to the Students model introduced above.

We illustrate the form of the posterior possibility function $f_{\bm{r}}(\cdot \given y_{1:n})$ in the case where $Y$ is a $\mathrm{N}(1, 1)$-distributed random variable and where $Y'$ is independent of $Y$ and is $\mathrm{N}(1/100, 1/100)$-distributed. This is a difficult case since $Y'$ is close to $0$ with a standard deviation allowing values of either signs. In order to show how the posterior possibility function $f_{\bm{r}}(\cdot \given y_{1:n})$ depends on $n$ and on the specific realisations $y_{1:n}$, two different values of $n$ and $1000$ realisations of the observations are simulated and the mean and standard deviation of the corresponding posterior possibility function $f_{\bm{r}}(\cdot \given y_{1:n})$ is displayed in Figure~\ref{fig:posteriorPossibilityRatio} together with the mean MAP for the Students model as well as for the normal model. This type of possibility function cannot be renormalised to a proper probability distribution since it is not integrable. Indeed, large values of $\bm{r}$ cannot be excluded since the credibility of the event $\bm{\mu}' = 0$ is positive for finite values of $n$.

\begin{figure}
\centering
    \begin{subfigure}[b]{0.49\textwidth}
      \includegraphics[width=\textwidth,trim=10pt 0pt 35pt 15pt,clip]{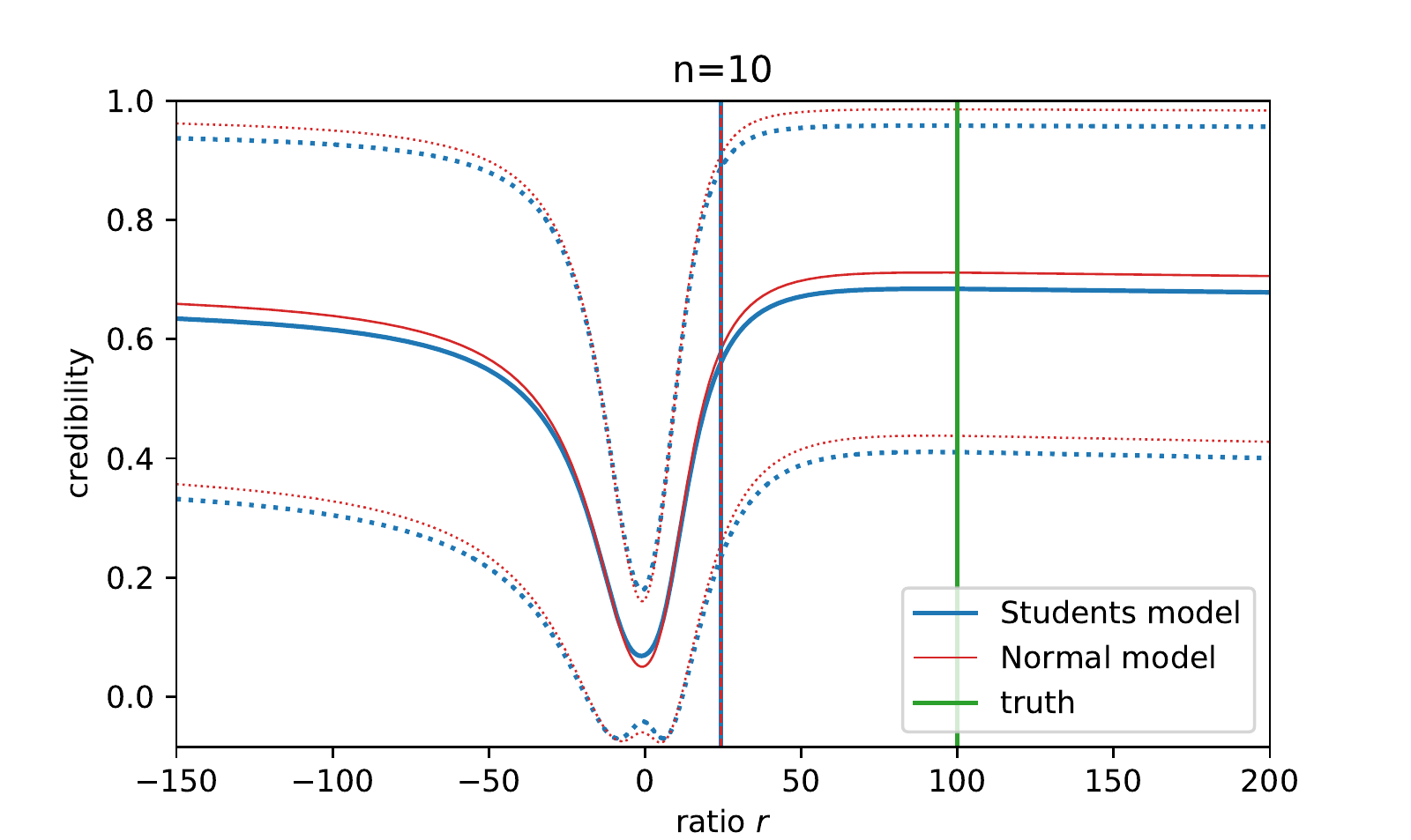}
    \end{subfigure} 
    \begin{subfigure}[b]{0.49\textwidth}
      \includegraphics[width=\textwidth,trim=10pt 0pt 35pt 15pt,clip]{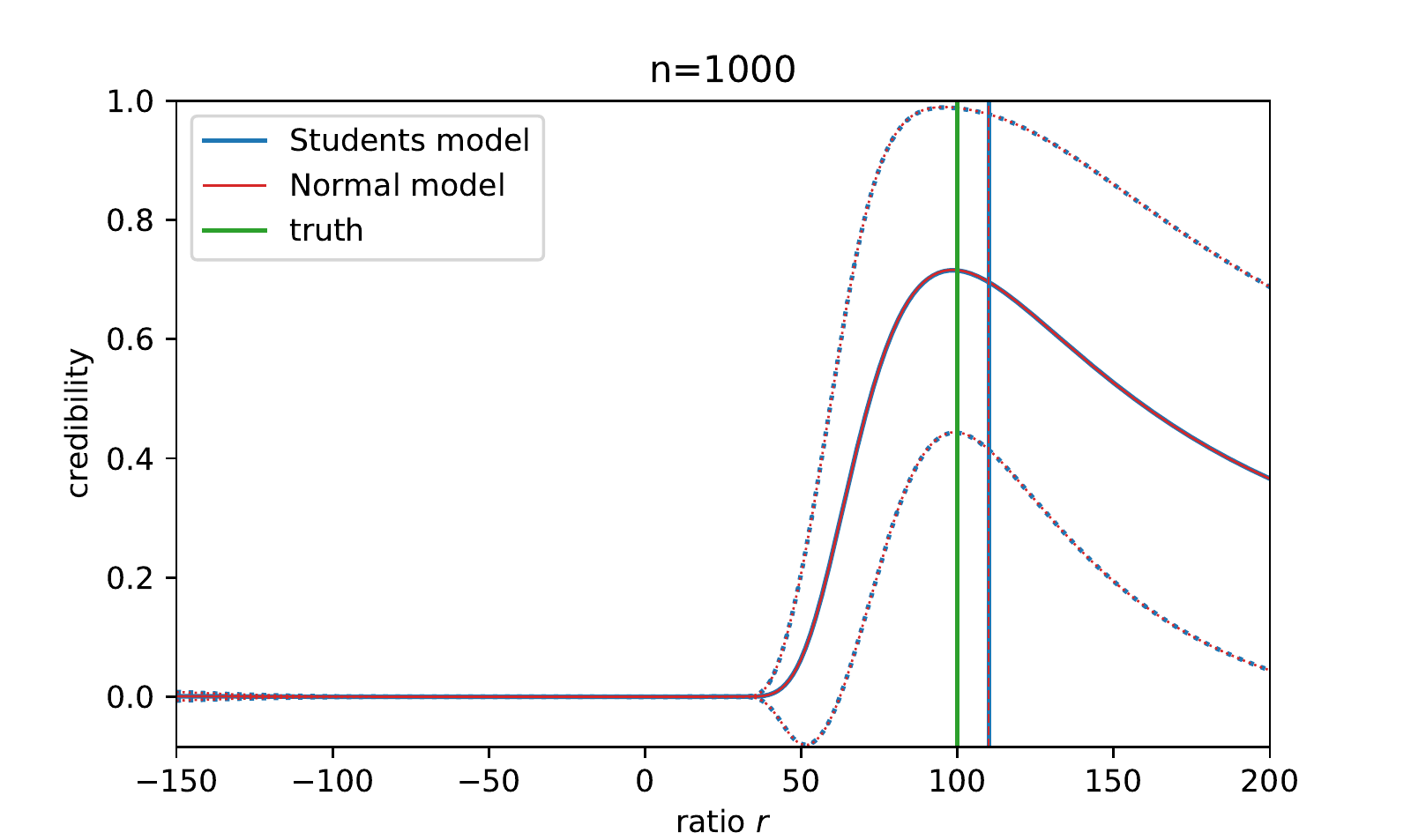}
    \end{subfigure}
\caption{Averaged posterior possibility function $f_{\bm{r}}(\cdot \given y_{1:n})$ (solid lines) $\pm$ one standard deviation (dashed lines) for a varying number of observations and averaged MAP (blue/red vertical lines) for the Students model and the normal model.}
\label{fig:posteriorPossibilityRatio}
\end{figure}

\section{Conclusion}

Through a number of practical and theoretical results, we have demonstrated how possibility functions can be used for statistical modelling in place of, or together with, probability distributions. Many features of standard probabilistic inference are retained when using possibility functions and the associated uncertain variables. The differences between the two approaches highlight the suitability of each representation of uncertainty for specific tasks. For instance, statistical modelling based on possibility functions enables the representation of a complete absence of prior information about some or all of the parameters of the model while avoiding the issue of improper posterior distributions. Crucially, although the influence of the information in the prior representation of uncertainty vanishes as the number of data points increases, the nature of the prior as a probability distribution or as a possibility function remains in the limit and influences the posterior quantification of uncertainty.

An interesting direction to take from this work is to develop an outer-measure approach in an infinite-dimensional setting. By doing so one can consider different avenues for uncertainty quantification, with applications in inverse problems and compressed sensing \citep{HGS12}. In particular, due to the recent success of the Bayesian approach \citep{DLSV13,EHNR96,AMS10} for inverse problems one could use this as an instinctive application. Another direction could be the derivation of asymptotic properties for dynamical systems \citep{BRR98,MMNP15} in the context of o.p.m., which is not considered in the existing work on this topic \citep{JH18b,HB18}.

\bibliographystyle{abbrvnat}
\bibliography{outer_MLE}

\appendix

\section{Additional results and proofs for Section~\ref{sec:CLT}}
\label{sec:proofsLLN}

\subsection{Expected value}

Let $\uvx$ and $\uvz$ be uncertain variables on $\mathcal{X}$ and $\mathcal{Z}$ respectively, jointly described by a possibility function $f_{\uvx,\uvz}$ on $\mathcal{X} \times \mathcal{Z}$. Let $\bar{\mathbb{P}}$ be the o.p.m.\ on $\Omega_{\mathrm{u}}$ induced by $f_{\uvx,\uvz}$. The conditional o.p.m.\ $\bar{\mathbb{P}}( \cdot \given \uvz = z)$ is characterised by
\begin{equation*}
\bar{\mathbb{P}}( A \given \uvz = z) = \dfrac{\bar{\mathbb{P}}(A \cap \{\uvz = z\})}{\bar{\mathbb{P}}(\{\uvz = z\})},
\end{equation*}
for any $A \subseteq \Omega_{\mathrm{u}}$ and any $z \in \mathcal{Z}$ such that $\bar{\mathbb{P}}(\{\uvz = z\}) > 0$. A conditional version $\mathbb{E}^*( \cdot \given \uvz = z)$ of the expectation $\mathbb{E}^*(\cdot)$ can then be introduced based on $\bar{\mathbb{P}}( \cdot \given \uvz = z)$ as follows:
\begin{equation*}
\mathbb{E}^*(\uvx \given \uvz = z) = \uvx\Big( \argmax_{\omega \in \Omega_{\mathrm{u}}} \bar{\mathbb{P}}( \omega \given \uvz = z) \Big).
\end{equation*}
We will write $\mathbb{E}^*(\uvx \given z)$ when there is no ambiguity.

\begin{prop}
\label{res:surjectiveTransformArgmax}
Let $\uvx$ be an uncertain variable on a set $\mathcal{X}$ and let $\xi$ be a surjective function from a set $\mathcal{Z}$ to $\mathcal{X}$. It holds that
\begin{equation*}
\argmax_{x \in \mathcal{X}} f_{\uvx}(x) = \xi\bigg( \argmax_{z \in \mathcal{Z}} f_{\uvx}(\xi(z)) \bigg).
\end{equation*}
\end{prop}

\begin{proof}
By definition of the argmax it holds that
\begin{align*}
\xi\bigg( \argmax_{z \in \mathcal{Z}} f_{\uvx}(\xi(z)) \bigg) & = \xi\big( \{ z : f_{\uvx}(\xi(z)) \geq f_{\uvx}(\xi(z')), z' \in \mathcal{Z}  \} \big) \\
& = \{ \xi(z) : f_{\uvx}(\xi(z)) \geq f_{\uvx}(\xi(z')), z' \in \mathcal{Z}  \}.
\end{align*}
By defining $x = \xi(z)$ and by noticing that $\{f_{\uvx}(\xi(z')) : z' \in \mathcal{Z}\} = \{ f_{\uvx}(x') : x' \in \mathcal{X} \}$ holds from the fact that $\xi$ is surjective, it follows that
$$
\xi\bigg( \argmax_{z \in \mathcal{Z}} f_{\uvx}(\xi(z)) \bigg) = \{ x : f_{\uvx}(x) \geq f_{\uvx}(x'), x' \in \mathcal{X} \},
$$
which concludes the proof of the proposition.
\end{proof}

\subsection{Proofs of the results in Section~\ref{sec:CLT}}

\begin{proof}[Proof of Proposition~\ref{res:chiSquared}]
The possibility function $f_{\uvpsi}$ describing the uncertain variable $\uvpsi$ is characterised by
\begin{align*}
f_{\uvpsi}(\psi) & = \sup\bigg\{ \prod_{i=1}^n \overline{\mathrm{N}}(\theta_i;\mu_i,\sigma^2) : \theta_1^2 + \dots + \theta_n^2 = \psi \bigg\} \\
& = \sup\bigg\{  \exp\bigg(-\frac{1}{2\sigma^2}\bigg( \psi - 2 \sum_{i=1}^n \theta_i \mu_i + \bar{\mu} \bigg) \bigg) : \theta_1^2 + \dots + \theta_n^2 = \psi \bigg\},
\end{align*}
for any $\psi \geq 0$. If $\psi = 0$ then the result is trivial. To solve this maximization problem when $\psi > 0$, we define the Lagrange function
\begin{equation*}
\mathcal{L}(\theta_1,\dots,\theta_n,\lambda) = -\psi + 2 \sum_{i=1}^n \theta_i \mu_i - \bar{\mu} - \lambda \bigg(\sum_{i=1}^n \theta_i^2 - \psi \bigg),
\end{equation*}
which easily yields $\theta_i = \dfrac{\mu_i}{\nu}$, $i = 1,\dots,n$ and $\lambda = \sqrt{\bar{\mu}}/\sqrt{\psi}$, so that $\theta_i = \mu_i \sqrt{\psi}/\sqrt{\bar{\mu}}$ and
\begin{equation*}
f_{\uvpsi}(\psi) = \exp\Big(-\dfrac{1}{2\sigma^2} \big(\psi - 2\sqrt{\psi\bar{\mu}} + \bar{\mu} \big) \Big),
\end{equation*}
for any $\psi \geq 0$, which concludes the proof of the lemma.
\end{proof}

\begin{proof}[Proof of Theorem~\ref{thm:LLN}]
Before proceeding with the proof, we emphasize an abuse of notation where now $y$ denotes a point in $\mathcal{X} = \mathbb{R}^d$, unrelated to previous notations in the document. We also write $f$ instead of $f_{\uvx}$ for the sake of simplicity.
\begin{description}[wide]
  \item[Definitions]
  We denote by $\mathcal{C}_f = \Conv(\argmax f)$ the convex hull of $\argmax f$, and by $d_{\mathcal{C}_f}$ the distance to the set $\mathcal{C}_f$, i.e. the function
  \begin{equation*}
	d_{\mathcal{C}_f }(x) = \inf_{y \in \mathcal{C}_f} \norm{x - y},\quad x \in \mathcal{X}.
  \end{equation*}
  Since $\mathcal{C}_f$ is convex, for any $x \in \mathcal{X}$, there exists a unique point $x^{\mathcal{C}_f} \in \mathcal{C}_f$ such that $d_{\mathcal{C}_f}(x) = \norm{x - x^{\mathcal{C}_f}}$. We also recall, from the definition of $\uvs_n$, that
  \begin{equation} \label{eq:fsn}
	f_{\uvs_n}(y) = \sup \left\{\prod_{i = 1}^n f(x_i) \GGiven  n^{-1}\sum_{i =1}^n x_i = y\right\},\quad n \in \mathbb{N}.
  \end{equation}
  \item[Outline]
  We aim to prove that
  \begin{equation*}
	\lim_{n \to \infty} f_{\uvs_n}(y) =
	\begin{dcases}
	  1,\quad&\textrm{if~}y \in \mathcal{C}_f,
	  \\
	  0,&\textrm{otherwise.}
	\end{dcases}
  \end{equation*}
  We will consider the two cases above separately.  
  \item[Case $y \in \mathcal{C}_f$:]
  The result being evident on $\argmax f$, let $y \in \mathcal{C}_f \setminus \argmax f$ be an arbitrary point on the convex hull $\mathcal{C}_f$ that does not belong to $\argmax f$.
  
  Using Carath\'{e}odory's theorem, $y$ can be written as the convex combination of at most $d + 1$ points of $\argmax f$, i.e., there exists $2 \leq p \leq d + 1$ such that
  \begin{equation*}
	y = \sum_{i = 1}^p c_i a_i,
  \end{equation*}
  where $c_i > 0$ and $a_i \in \argmax f$, $1 \leq i \leq p$, with $\sum_1^p c_i = 1$.

  For any $n \geq \max_i\{c_i^{-1}\} + 1$, we consider the sequence of points $(x_{n, i})_{i = 1}^n \in \mathcal{X}^n$ defined as
  \begin{equation*}
	x_{n, i} =
	\begin{dcases}
	  \dfrac{c_in}{\floor{c_in}}a_i, &\textrm{if}\quad \sum_{j = 1}^{p' - 1}\floor{c_jn} + 1 \leq i \leq \sum_{j = 1}^{p'}\floor{c_jn},\quad 1 \leq p' \leq p - 1,
	  \\
	  \dfrac{c_pn}{n - \sum_{j = 1}^{p - 1}\floor{c_jn}}a_p, &\textrm{if} \quad \sum_{j = 1}^{p - 1}\floor{c_jn} + 1 \leq i \leq n,
	\end{dcases}
  \end{equation*}
  and one can easily verify that $y = n^{-1}\sum_{i = 1}^n x_{n, i}$. We can also write
  \begin{subequations}
	\begin{align}
	  \prod_{i = 1}^n f(x_{n, i}) &= \prod_{i = 1}^{p - 1} f\left(\dfrac{c_in}{\floor{c_in}}a_i\right)^{\floor{c_in}} f\left(\dfrac{c_pn}{n - \sum_{j = 1}^{p - 1}\floor{c_jn}}a_p\right)^{n - \sum_{j = 1}^{p - 1}\floor{c_jn}} \;
	  \\
	  &= \prod_{i = 1}^p f\left(\left(1 + \dfrac{\alpha_{n, i}}{\beta_{n, i}}\right)a_i\right)^{\beta_{n, i}}, \label{eq:prod_f_xni_b}
	\end{align}
  \end{subequations}
  where $$\alpha_{n, i} = c_in - \floor{c_in} \in [0,1)\quad\textrm{and}\quad\beta_{n, i} = \floor{c_in}\geq c_in - 1,\quad 1 \leq i \leq p - 1,$$
  and $$\alpha_{n, p} = \sum_{j = 1}^{p - 1}(\floor{c_jn} - c_jn) \in (1 - p,0]\quad\textrm{and}\quad\beta_{n, p} = n - \sum_{j = 1}^{p - 1}\floor{c_jn} \geq c_pn,$$ so that $\lim_n \alpha_{n, i}/\beta_{n, i} = 0$ for any $1 \leq i \leq p$.
  
  For any $1 \leq i \leq p$, then, since $f$ attains its supremum value $1$ in $a_i$ and $f$ is $\mathcal{C}^2$ in some open neighbourhood of $a_i$, Taylor's theorem yields
  \begin{equation*}
	f\left(\left(1 + \dfrac{\alpha_{n, i}}{\beta_{n, i}}\right)a_i\right) = 1 + \dfrac{1}{2}\dfrac{\alpha_{n, i}^2}{\beta_{n, i}^2}a_i^{\mathrm{t}}H_f(a_i)a_i + o\left(\dfrac{1}{2}\dfrac{\alpha_{n, i}^2}{\beta_{n, i}^2}\|a_i\|^2\right),
  \end{equation*}
  where $H_f(a_i)$ is the Hessian matrix of $f$ in $a_i$. That is,
  \begin{align*}
	f\left(\left(1 + \dfrac{\alpha_{n, i}}{\beta_{n, i}}\right)a_i\right)^{\beta_{n, i}} \!\!\!\! &= \exp\left[\beta_{n, i}\log\left(1 + \dfrac{1}{2}\dfrac{\alpha_{n, i}^2}{\beta_{n, i}^2}a_i^{\mathrm{t}}H_f(a_i)a_i + o\left(\dfrac{1}{2}\dfrac{\alpha_{n, i}^2}{\beta_{n, i}^2}\|a_i\|^2\right)\right) \right]
	\\
	&\sim_n \exp\left[\dfrac{1}{2}\dfrac{\alpha_{n, i}^2}{\beta_{n, i}}a_i^{\mathrm{t}}H_f(a_i)a_i + o\left(\dfrac{1}{2}\dfrac{\alpha_{n, i}^2}{\beta_{n, i}}\|a_i\|^2\right) \right],
  \end{align*}
  so that $\lim_n f\left(\left(1 + \dfrac{\alpha_{n, i}}{\beta_{n, i}}\right)a_i\right)^{\beta_{n, i}} = 1$.
  
  From \eqref{eq:prod_f_xni_b} it holds that $\lim_n \prod_{i = 1}^n f(x_{n, i}) = 1$, and from \eqref{eq:fsn} if follows that $\lim_{n} f_{\uvs_n}(y) = 1$.
  \item[Case $y \notin \mathcal{C}_f$:]
  Let $y \in \mathcal{X} \setminus \mathcal{C}_f$ be an arbitrary point outside the convex hull $\mathcal{C}_f$, and let us denote by $\delta = d_{\mathcal{C}_f}(y) > 0$ its distance to $\mathcal{C}_f$. We define the open set
  \begin{equation*}
	B_0 = \{x \in \mathcal{X} \given  d_{\mathcal{C}_f}(x) < \delta/2\},
  \end{equation*}
  and the sequence of increasing closed sets $\{B_n\}_{n \in \mathbb{N}^*}$ as 
  \begin{equation*}
	B_n = \{x \in \mathcal{X} \given  \delta/2 \leq d_{\mathcal{C}_f}(x) \leq \delta(1 + \sqrt n )\},\quad n \in \mathbb{N}^*.
  \end{equation*}
  We define $b_n = \sup_{x \in B_n} f(x)$ and $\bar{b}_n = \sup_{x \in \mathcal{X} \setminus (B_n \cup B_0)} f(x)$, $n \in \mathbb{N}$, and we note that
  \begin{itemize}[label=-]
	\item Since $f$ is bounded and continuous and the sets $\{B_n\}_{n \in \mathbb{N}^*}$ are closed, the supremums $b_n$ are all reached and we have $0 \leq b_n < 1$ for $n \geq 1$. We define $b_{*} = \sup_{n \geq 1}b_n$ and we have $0 \leq b_{*} < 1$.
	\item Since $\lim_{\norm{x} \to \infty} f(x) = 0$, $\lim_{n \to \infty} \bar{b}_n = 0$.
  \end{itemize}
  Recall from \eqref{eq:fsn} that, for any $n \in \mathbb{N}$, we need to consider the sequences of points $x_{1:n}$ satisfying $n^{-1}\sum_{i =1}^n x_i = y$. We will focus first on the set $$Y_n = \{x_{1:n} \in \mathcal{X}^n \given  x_1,\dots,x_n \in B_n \cup B_0, n^{-1}\sum_{i =1}^n x_i = y\},$$ i.e., the admissible sequences whose points are all contained within $B_n \cup B_0$, and then on the set
$$
\bar{Y}_n = \{x_{1:n} \in \mathcal{X}^n \given  n^{-1}\sum_{i =1}^n x_i = y\} \setminus Y_n,
$$
i.e, those with a least one point in the remaining space $\mathcal{X}\setminus B_n \cup B_0$.
 
  Denote by $\hat n = \min_{x_{1:n} \in Y_n} \sum_{i = 1}^n \bm{1}_{B_n}(x_i)$ the minimum number of points in $B_n$ across every sequence in $Y_n$, and consider a sequence $\hat{x}_{1:n} \in Y_n$ with $\hat n$ points in $B_n$, indexed from $1$ to $\hat n$. Since $\hat{x}^{\mathcal{C}_f}_i \in \mathcal{C}_f$ for any $1 \leq i \leq n$ and $\mathcal{C}_f$ is convex, we have $\frac{1}{n} \sum_{i = 1}^n \hat{x}^{\mathcal{C}_f}_i \in \mathcal{C}_f$. We may then write
  \begin{align*}
	\delta &= d_{\mathcal{C}_f}(y)
	\\
	&\leq \norm{y - \dfrac{1}{n} \sum_{i = 1}^n \hat{x}^{\mathcal{C}_f}_i}
	\\
	&= \norm{\dfrac{1}{n} \sum_{i = 1}^n \hat{x}_i - \dfrac{1}{n} \sum_{i = 1}^n \hat{x}^{\mathcal{C}_f}_i}
	\\
	&\leq \dfrac{1}{n} \sum_{i = 1}^{\hat n} \norm{\hat{x}_i - \hat{x}^{\mathcal{C}_f}_i} + \dfrac{1}{n} \sum_{i = \hat n + 1}^{n} \norm{\hat{x}_i - \hat{x}^{\mathcal{C}_f}_i}
	\\
	&\leq \dfrac{\hat n}{n}\delta(1 + \sqrt n) + \dfrac{n - \hat n}{n}\delta/2.
  \end{align*}
  It follows that $\frac{n}{1 + 2\sqrt n} \leq \hat n$, and thus $\lim_{n} \hat n = \infty$. However, since
  \begin{equation*}
	\sup_{x_{1:n} \in Y_n} \prod_{i = 1}^n f(x_i) \leq b_n^{\hat n} \leq b_*^{\hat n},
  \end{equation*}
  and $b_* < 1$, it follows that $\lim_{n} \sup_{x_{1:n} \in Y_n} \prod_{i = 1}^n f(x_i) = 0$.
  
  Any sequence in $\bar{Y}_n$ has at least one point in $\mathcal{X} \setminus (B_n \cup B_0)$, and thus
  \begin{equation*}
	\sup_{x_{1:n} \in \bar{Y}_n} \prod_{i = 1}^n f(x_i) \leq \bar{b}_n.
  \end{equation*}
  Since $\lim_{n} \bar{b}_n = 0$, it follows that $\lim_{n} \sup_{x_{1:n} \in \bar{Y}_n} \prod_{i = 1}^n f(x_i) = 0$.
  
  We can then write
  \begin{align*}
	\lim_{n \to \infty} f_{\uvs_n}(y) & = \lim_{n \to \infty} \sup_{x_{1:n} \in Y_n \cup \bar{Y}_n} \prod_{i = 1}^n f(x_i) \\
& = \lim_{n \to \infty} \max\bigg\{ \sup_{x_{1:n} \in Y_n} \prod_{i = 1}^n f(x_i), \sup_{x_{1:n} \in \bar{Y}_n} \prod_{i = 1}^n f(x_i) \bigg\},
  \end{align*}
  and since for two convergent sequences $\{u_n\}_n$ and $\{v_n\}_n$ it holds that
  \begin{equation*}
  \lim_{n\to\infty} \max\{u_n,v_n\} = \max\Big\{\lim_{n\to\infty} u_n, \lim_{n\to\infty} v_n\Big\},
  \end{equation*}
  it follows that $\lim_{n} f_{\uvs_n}(y) = 0$.
\end{description}
\end{proof}

\begin{proof}[Proof of Theorem~\ref{thm:CLT}]
One of the basic properties of strictly log-concave functions is that they are maximised at a single point which we denote by $\mu$. We assume without loss of generality that $\mu = 0$ and We write $f$ instead of $f_{\uvx}$ for the sake of simplicity. To find the supremum of $\prod_{i=1}^n f(x_i)$ over the set of $x_i$'s verifying $n^{-1/2}\sum_{i=1}^n x_i = x$, we first use Lagrange multipliers to find that $$f'(x_i)f(x_j) = f'(x_j)f(x_i),$$ for any $i,j \in \{1,\dots,n\}$ so that a solution is $x_i = n^{-1/2}x$. In order to show that this solution is local maximizer, we consider the bordered Hessian corresponding to our constrained optimisation problem, defined as
$$
H =
\begin{bmatrix}
0          & 1/\sqrt{n} &            & \dots  &   & 1/\sqrt{n} \\
1/\sqrt{n} & a          & b          & \dots  &   & b \\
           & b          &            &        &   &  \\
\vdots     & \vdots     &            & \ddots &   &  \vdots \\
           &            &            &        &   & b \\
1/\sqrt{n} & b          &            & \dots  & b & a 
\end{bmatrix},
$$
where $a = f''(y)f(y)^{n-1}$ and $b = f'(y)^2f(y)^{n-2}$ with $y = x/\sqrt{n}$. For the solution $x_i = y$, $i \in \{1,\dots,n\}$, to be a local maximum, the sign of the principal minors $M_3, \dots, M_n$ of $H$ has to be alternating, starting with $M_3$ positive. Basic matrix manipulations for the determinant yield
\begin{equation}
\label{eq:minors}
M_k = -\dfrac{k-1}{n}(a - b)^{k-2},
\end{equation}
which is alternating in sign. For $M_3$ to be positive, it has to hold that
$$
f''(y)f(y) < f'(y)^2.
$$
This condition can be recognized as a necessary and sufficient condition for a function to be strictly log-concave. It also follows from the assumption of log-concavity that the condition $f'(x_i)f(x_j) = f'(x_j)f(x_i)$, which can be expressed as $(\log f(x_i))' = (\log f(x_j))'$ can only be satisfied at $x_i = x_j$ so that this solution is a global maximum. We therefore study the behaviour of the function $f(\frac{x}{\sqrt{n}})^n$ as $n \to \infty$ and obtain
$$
f\bigg(\dfrac{x}{\sqrt{n}}\bigg)^n = \exp\bigg( f'(0) \sqrt{n}x + \dfrac{1}{2} \big(f''(0) - f'(0)^2\big)x^2 + O\big(n^{-1/2}\big) \bigg).
$$
The result of the proposition follows easily by taking the limit and by noting that $f'(0) = 0$ and that $f''(0)$ is non-positive since $f$ decreases in the neighbourhood of its $\argmax$.
\end{proof}

\section{Additional results and proofs for Section~\ref{sec:MLE_un}}
\label{sec:additionalConcepts}

We begin this section with a useful result regarding the evolution of the variance when transforming uncertain variables.

\begin{prop}
\label{prop:mapVariance}
Let $\uvx$ be an uncertain variable on a set $\mathcal{X}$ with expected value $\mu = \mathbb{E}^*(\uvx)$ described by the possibility function $f_{\uvx}$, let $\xi : \mathcal{X} \to \mathcal{Z}$ be a bijective function in another set $\mathcal{Z}$ such that both $\xi$ and $\xi^{-1}$ are twice differentiable and let $\uvz$ be the uncertain variable $\xi(\uvx)$, then $\uvz$ is described by $f_{\uvx} \circ \xi^{-1}$ and the variance of $\uvz$ can be expressed as
\begin{equation*}
\mathbb{V}^*(\uvz) = (\partial_x \xi(\mu))^2 \mathbb{V}^*(\uvx).
\end{equation*}
\end{prop}

\begin{proof}[Proof of Proposition~\ref{prop:mapVariance}]
Since $\xi$ is bijective, the possibility function $f_{\uvz}$ describing $\uvz$ is indeed
$$
f_{\uvz}(z) = \sup \big\{ f_{\uvx}(x) : x \in \mathcal{X}, \xi(x) = z \big\} = (f_{\uvx} \circ \xi^{-1})(z),
$$
for any $z \in \mathcal{Z}$. The variance of $\uvz$ is 
\begin{equation*}
\mathbb{V}^*(\uvz) = -\dfrac{1}{\partial_z^2 f_{\uvz}(\mathbb{E}^*(\uvz))}.
\end{equation*}
The second derivative in the denominator can be computed as
\begin{equation*}
\partial_z^2 f_{\uvz}(z) = \partial_z^2 \xi^{-1}(z) \partial_x f_{\uvx}(\xi^{-1}(z)) + (\partial_z \xi^{-1}(z))^2 \partial_x^2 f_{\uvx}(\xi^{-1}(z)).
\end{equation*}
Using the fact that $\mathbb{E}^*(\uvz) = \xi(\mu)$ as well as the standard rules for the derivative of the inverse of a function, it follows that
\begin{equation*}
\partial_z^2 f_{\uvz}(\mathbb{E}^*(\uvz)) = \dfrac{1}{(\partial_x \xi(\mu))^2} \partial_x^2 f_{\uvx}(\mu).
\end{equation*}
The result of the lemma follows from the identification of the term $-1/\partial_x^2 f_{\uvx}(\mu)$ as the variance of $\uvx$.
\end{proof}

\subsection{Proofs of the results in Section~\ref{sec:MLE_un}}

Some additional technical concepts and results are needed to state the proofs of the results in Section~\ref{sec:MLE_un}. We say that the sequence $(\uvx_1, \uvx_2, \dots)$ converges in \emph{credibility} to an uncertain variable $\uvx$ in $\mathcal{X}$ if for all $\delta > 0$
\begin{equation}
\label{eq:convCredibility}
\lim_{n \to \infty} \bar{\mathbb{P}} \big( |\uvx_n - \uvx| > \delta \big) = 0.
\end{equation}
This is denoted $\uvx_n \xrightarrow{c.} \uvx$. The LLN for uncertain variables can be proved to give a convergence in credibility to the expected value. Using these two concepts of convergence, we state without proving the analogue of Slutsky's lemma as follows.

\begin{prop}
\label{prop:slutsky}
Let $(\uvx_1, \uvx_2, \dots)$ and $(\uvz_1,\uvz_2,\dots)$ be two sequences of uncertain variables. If it holds that $\uvx_n \xrightarrow{o.p.m.} \uvx$ and $\uvz_n \xrightarrow{c.} \alpha$ for some uncertain variable $\uvx$ and some constant $\alpha$ then
\begin{equation*}
\uvx_n + \uvz_n \xrightarrow{o.p.m.} \uvx + \alpha, \quad \uvx_n\uvz_n \xrightarrow{o.p.m.} \alpha\uvx,\quad \text{and} \quad \uvx_n/\uvz_n \xrightarrow{o.p.m.} \uvx/\alpha,
\end{equation*}
given that $\alpha$ is invertible.
\end{prop}

We note that the proof of Slutsky's lemma for uncertain variables is beyond the scope of this work.

\begin{proof}[Proof of Theorem~\ref{res:BvM}]
Following the standard approach, we introduce the possibility function describing $\uvpsi = \sqrt{n}(\uvtheta - \theta_0)$ as
\begin{align*}
f_{\uvpsi}(\psi \given  y_1,\dots, y_n) & = \sup \bigg\{ \dfrac{L_n(\theta)f_{\uvtheta}(\theta)}{\sup_{\theta' \in \Theta} L_n(\theta')f_{\uvtheta}(\theta')} : \theta \in \Theta, \sqrt{n}(\theta - \theta_0) = \psi \bigg\} \\
& = \dfrac{L_n(\theta_0 + \psi/\sqrt{n})f_{\uvtheta}(\theta_0 + \psi/\sqrt{n})}{\sup_{\psi' \in \Theta} L_n(\theta_0 + \psi'/\sqrt{n})f_{\uvtheta}(\theta_0 + \psi'/\sqrt{n})}.
\end{align*}
Again, as usual, we consider large values of $n$ within the argument of the prior only. Under Assumption~\ref{it:priorDiff}, noticing that the argument of the supremum in the denominator is maximized at the MLE leads to
$$
f_{\uvpsi}(\psi \given y_{1:n}) \approx \dfrac{L_n(\theta_0 + \psi/\sqrt{n})}{L_n(\hat{\theta}_n)}.
$$
where the MLE $\hat{\theta}_n$ is deterministic since the observations are given. As opposed to the corresponding proof in a probabilistic context, the MLE appears naturally in the approximation of $f_{\uvtheta}(\cdot \given y_{1:n})$, this is due to the supremum in the denominator of Bayes' rule for possibility functions. This difference suggests that the BvM theorem is indeed very natural with possibility functions. In order to relate $\hat{\theta}_n$ to the true parameter $\theta_0$, we expand the derivative of the log-likelihood $\ell_n$ at $\theta_0$ around $\hat{\theta}_n$ as
\begin{equation*}
\partial_{\theta} \ell_n(\theta_0) = \partial_{\theta} \ell_n(\hat{\theta}_n) + (\theta_0 - \hat{\theta}_n) \partial_{\theta}^2 \ell_n(\hat{\theta}_n ) + \dfrac{1}{2}(\theta_0 - \hat{\theta}_n)^2 \partial_{\theta}^3 \ell_n(\psi_n ),
\end{equation*}
for some $\psi_n$ in the interval formed by $\hat{\theta}_n$ and $\theta_0$. Since it holds that $\partial_{\theta} \ell_n(\hat{\theta}_n) = 0$ by construction, it follows that
\begin{equation}
\label{eq:distTrueMLE}
\sqrt{n}(\hat{\theta}_n - \theta_0) = \dfrac{\frac{1}{\sqrt{n}} \partial_{\theta} \ell_n(\theta_0 )}{-\frac{1}{n}\partial_{\theta}^2 \ell_n(\hat{\theta}_n) + \frac{1}{2n}(\hat{\theta}_n - \theta_0)\partial_{\theta}^3 \ell_n(\psi_n)}.
\end{equation}
The term $-n^{-1}\partial_{\theta}^2 \ell_n(\hat{\theta}_n)$ is the observed information which we denote $\mathcal{J}^*_n$. It follows from Assumption~\ref{it:identifiable} that the second term in the denominator of \eqref{eq:distTrueMLE} vanishes for $n$ large. The MLE therefore verifies
\begin{equation}
\label{eq:MleAsFunctionOfDelta}
\sqrt{n}(\hat{\theta}_n - \theta_0) = \Delta_n \doteq \sqrt{n}\dfrac{\partial_{\theta} \ell_n(\theta_0)}{\mathcal{J}^*_n}.
\end{equation}
which yields
$$
f_{\uvpsi}(\psi \given y_{1:n}) \approx \dfrac{L_n(\theta_0 + \psi/\sqrt{n})}{L_n(\theta_0 + \Delta_n/\sqrt{n})} = \exp\big( \ell_n(\theta_0 + \psi/\sqrt{n}) - \ell_n(\theta_0 + \Delta_n/\sqrt{n}) \big).
$$
The expression of the posterior possibility function $f_{\uvpsi}(\cdot \given  y_{1:n})$ can be further simplified by expanding the log-likelihood $\ell_n$ around $\hat{\theta}_n$ as
\begin{equation*}
\ell_n(\theta) = \ell_n(\hat{\theta}_n) + \dfrac{1}{2} (\theta - \hat{\theta}_n)^2 \partial_{\theta}^2 \ell_n(\hat{\theta}_n) + \dfrac{1}{6} (\theta - \hat{\theta}_n)^3 \partial_{\theta}^3 \ell_n(\psi'_n),
\end{equation*}
with $\psi'_n$ also lying between $\theta$ and $\theta_0$ and by considering $\theta = \theta_0 + \psi/\sqrt{n}$ and $\theta = \theta_0 + \Delta_n/\sqrt{n}$. Using Assumption~\ref{it:identifiable} once more, we deduce the following approximation
\begin{equation}
\label{eq:posteriorPsi}
f_{\uvpsi}(\psi \given y_{1:n}) \approx \exp\bigg( -\dfrac{\mathcal{J}^*_n}{2n}(\psi - \Delta_n)^2 \bigg).
\end{equation}
Returning to the possibility function describing $\uvtheta$ a posteriori and using the relation \eqref{eq:MleAsFunctionOfDelta} once more, leads to the desired result.
\end{proof}

\begin{proof}[Proof of Theorem~\ref{thm:asympNormUnknownSamp}]
We make explicit the dependency of $\Delta_n$ and $\mathcal{J}^*_n$ on the observation by writing $\Delta_n(\uvy_{1:n})$ and $\mathcal{J}^*_n(\uvy_{1:n})$ instead. It follows from Theorem~\ref{res:BvM} and from Assumption~\ref{it:bounded2} that
$$
\sqrt{n}(\uvtheta^*_n - \theta_0) = \dfrac{n}{\mathcal{J}^*_n(\uvy_{1:n})} \times \dfrac{1}{\sqrt{n}} \sum_{i=1}^n \partial_{\theta} \ell(\theta_0; \uvy_i).
$$
The LLN for uncertain variables, together with Assumption~\ref{it:identifiable}, yields
\begin{equation}
\label{eq:convergenceObservedInfo}
\dfrac{1}{n}\mathcal{J}^*_n(\uvy_{1:n}) = \dfrac{1}{n} \sum_{i=1}^n \partial_{\theta}^2 \ell(\hat{\uvtheta}_n; \uvy_i) \xrightarrow{o.p.m.} \mathbb{E}^*\big( \partial_{\theta}^2 \ell(\theta_0; \uvy) \big) = \mathcal{I}^*(\theta_0),
\end{equation}
and the CLT for uncertain variables yields
\begin{equation}
\label{eq:convergenceScore}
\dfrac{1}{\sqrt{n}} \sum_{i=1}^n \partial_{\theta} \ell(\theta_0; \uvy_i) \xrightarrow{o.p.m.} \overline{\mathrm{N}}\big(0, \mathbb{V}^*(s_{\theta_0}(\uvy)) \big).
\end{equation}
The desired result follows from Slutsky's lemma for uncertain variables (Proposition~\ref{prop:slutsky}) and from the properties of normal uncertain variables.
\end{proof}

\begin{proof}[Proof of Theorem~\ref{thm:asymptotic_LRT_unknown}]
It follows from Theorem~\ref{res:BvM} and from Assumption~\ref{it:bounded2} that
\begin{equation*}
- 2 \log \lambda(\uvy_{1:n}) \approx \dfrac{1}{n}\mathcal{J}^*(\uvy_{1:n}) \Delta_n(\uvy_{1:n})^2 .
\end{equation*}
Expanding this expression based on the definition \eqref{eq:MleAsFunctionOfDelta} of $\Delta_n(\uvy_{1:n})$, we find that
\begin{equation*}
- 2 \log \lambda(\uvy_{1:n}) \approx \dfrac{n}{\mathcal{J}^*(\uvy_{1:n})} \bigg( \dfrac{1}{\sqrt{n}} \sum_{i=1}^n \partial_{\theta} \ell(\theta_0; \uvy_i) \bigg)^2.
\end{equation*}
The desired result follows from the convergence results \eqref{eq:convergenceObservedInfo} and \eqref{eq:convergenceScore} together with Slutsky's lemma for uncertain variables as well as with the definition and properties of the (non-central) chi-squared possibility function.
\end{proof}

\end{document}